\documentclass[reqno,oneside]{amsart}
\usepackage{hyperref}
\usepackage{geometry}
\usepackage[ansinew]{inputenc}
\usepackage{anyfontsize}
\usepackage{graphicx}
\usepackage{amsmath}% amstex ?
\usepackage{amsthm}
\usepackage{amssymb, color}%
\usepackage[numbers, square]{natbib}
\usepackage{mathrsfs}
\usepackage{bbm}
\usepackage{tikz}
\usepackage{enumitem}
\usepackage{mathptmx}

\usepackage{color}

  \evensidemargin
\oddsidemargin
\newcommand{\bD}{\mathbb{D}}

\newcommand{\1}{\mathbbm{1}}

\newcommand{\Zeros}{{\bf Zeros}\,}

\DeclareMathSymbol{\mlq}{\mathord}{operators}{``}
\DeclareMathSymbol{\mrq}{\mathord}{operators}{`'}

\newcommand{\E}{\mathbb E}

\newcommand{\R}{\mathbb{R}}
\newcommand{\N}{\mathbb{N}}

\newcommand{\C}{\mathbb{C}}
\newcommand{\Z}{\mathbb{Z}}

\renewcommand{\P}{\mathbb{P}}

\renewcommand{\Re}{\operatorname{Re}}
\renewcommand{\Im}{\operatorname{Im}}

\newcommand{\Var}{\mathop{\mathrm{Var}}\nolimits}

\newcommand{\eps}{\varepsilon}
\newcommand{\eqdistr}{\stackrel{d}{=}}

\newcommand{\todistr}{\overset{d}{\underset{n\to\infty}\longrightarrow}}
\newcommand{\todistrnon}{\overset{d}{\longrightarrow}}

\newcommand{\dd}{{\rm d}}
\newcommand{\eee}{{\rm e}}
\newcommand{\iii}{{\rm i}}

\usepackage[capitalise]{cleveref}

\theoremstyle{plain}
\newtheorem{theorem}{Theorem}[section]
\newtheorem{lemma}[theorem]{Lemma}
\newtheorem{corollary}[theorem]{Corollary}
\newtheorem{proposition}[theorem]{Proposition}

\theoremstyle{definition}

\theoremstyle{remark}
\newtheorem{remark}[theorem]{Remark}

\newcommand{\genstirlingII}[3]{%
  \genfrac{\{}{\}}{0pt}{#1}{#2}{#3}%
}

\newcommand{\stirling}[2]{\genfrac{[}{]}{0pt}{}{#1}{#2}}
\newcommand{\stirlingsec}[2]{\genfrac{\{}{\}}{0pt}{}{#1}{#2}}

\newcommand{\stirlingII}[2]{\genstirlingII{}{#1}{#2}}

\makeatletter
\def\@tocline#1#2#3#4#5#6#7{\relax
  \ifnum #1>\c@tocdepth % then omit
  \else
    \par \addpenalty\@secpenalty\addvspace{#2}%
    \begingroup \hyphenpenalty\@M
    \@ifempty{#4}{%
      \@tempdima\csname r@tocindent\number#1\endcsname\relax
    }{%
      \@tempdima#4\relax
    }%
    \parindent\z@ \leftskip#3\relax \advance\leftskip\@tempdima\relax
    \rightskip\@pnumwidth plus4em \parfillskip-\@pnumwidth
    #5\leavevmode\hskip-\@tempdima
      \ifcase #1
       \or\or \hskip 1em \or \hskip 2em \else \hskip 3em \fi%
      #6\nobreak\relax
    \dotfill\hbox to\@pnumwidth{\@tocpagenum{#7}}\par
    \nobreak
    \endgroup
  \fi}
\makeatother

\makeatletter
\@namedef{subjclassname@2020}{%
  \textup{2020} Mathematics Subject Classification}
\makeatother

\begin{document}

%\author{Marcel Fenzl}
%\address{Marcel Fenzl:
%Institut f\"ur Mathematik, Universit\"at Z\"urich
%Office Y27-J-22
%Winterthurerstrasse 190, CH-8057 Z\"urich, Switzerland
%}
%\email{marcel.fenzl@math.uzh.ch}

\author{Zakhar Kabluchko}
\address{Zakhar Kabluchko: Institut f\"ur Mathematische Stochastik,
Westf\"alische Wilhelms-Universit\"at M\"unster
%Orl\'eans-Ring 10,
%48149 M\"unster, Germany
}
\email{zakhar.kabluchko@uni-muenster.de}

\author{Alexander Marynych}
\address{Alexander Marynych: Faculty of Computer Science and Cybernetics, Taras Shevchenko National University of Kyiv}
\email{marynych@unicyb.kiev.ua}

\author{Helmut Pitters}
\address{Helmut Pitters:
Mathematisches Institut, Universit\"at Mannheim
}
\email{helmut.pitters@mail.uni-mannheim.de}

\title[Mod-\texorpdfstring{$\varphi$}{phi} convergence of Stirling distributions]{Mod-\texorpdfstring{$\varphi$}{phi} convergence of Stirling distributions and limit theorems for zeros of their generating functions}

%%%%%%%%%%%%%%%%%%For figures see "Touchard Lambert Simulations.nb"



\begin{abstract}
We study mod-$\varphi$ convergence of several probability distributions on the set of positive integers that involve Stirling numbers of both kinds and, as a consequence, derive various limit theorems for these distributions. We also derive closely related limit theorems for the distribution of zeros of the corresponding generating functions. For example, we identify the asymptotic distribution of zeros for the generating polynomial of the number of occupied boxes when $n$
balls are allocated equiprobably and independently among $\theta$ boxes in the regime when $\theta$ grows linearly with $n$.
% we prove mod-$\varphi$ convergence and compute the limit distribution of the zeros of the generating polynomials.
\end{abstract}

\maketitle

%\tableofcontents

\section{Introduction}

The \emph{Stirling cycle number} (or the \emph{Stirling number of the first kind}) $\stirling{n}{k}$ is defined as the number of permutations on $n\in\N$ elements that have exactly $k\in \{1,2,\ldots,n\}$ cycles. Alternatively, Stirling cycle numbers can be defined by their generating polynomial
$$
S_n(x) := x(x+1)(x+2) \ldots (x+n-1) = \sum_{k=1}^n \stirling{n}{k}x^k.
$$
The \emph{Stirling partition number} (or \emph{Stirling number of the second kind}) $\stirlingsec{n}{k}$ is defined as the number of partitions of a set of $n\in\N$ elements into $k\in \{1,\ldots,n\}$ non-empty subsets (blocks). We refer to the book~\cite[Section~6.1]{graham_knuth_patashnik_book} for an introduction to Stirling numbers of both kinds.

The generating functions of the Stirling partition numbers are called the \emph{Touchard polynomials} and denoted by
$$
T_n(x) = \sum_{k=1}^n \stirlingsec{n}{k} x^k.
$$
There are several equivalent definitions of these polynomials including the following ones:
\begin{equation}\label{eq:touchard_poly_def}
T_n(x)
=
\eee^{-x} \sum_{\ell=0}^\infty \frac{\ell^n x^\ell}{\ell!}
=
n! [t^n] \eee^{x (\eee^t-1)}
=
\eee^{-x} \left(x\frac{\dd }{\dd x}\right)^n \eee^x
=
B_n(x,\ldots,x),
\end{equation}
where $B_n(z_1,\ldots,z_n)$ is the $n$-th Bell polynomial. Touchard polynomials are also called single-variable Bell polynomials and have been studied by Ramanujan~\cite[Chapter~3]{berndt_book} before they were introduced by Touchard~\cite{touchard} and Bell~\cite{bell1934}; see~\cite[p.~48]{berndt_book} for a historical account. From the point of view of probability theory, the first identity in~\eqref{eq:touchard_poly_def} means that $T_n(x)$ is the $n$-th moment of the Poisson distribution with parameter $x>0$.

It was shown by Harper~\cite{harper} that all zeros of $T_n$ are real, nonpositive, and distinct. Using this observation, Harper~\cite{harper}  showed that the Stirling partition numbers satisfy a central limit theorem (CLT). Namely, consider a random variable $\xi_{n}$ with the following distribution:
\begin{equation}\label{eq:def_xi_n}
\P[\xi_{n}=k] = \frac 1 {B_n} \stirlingsec{n}{k} ,\qquad k\in \{1,\ldots,n\},
\end{equation}
where $B_n=T_n(1)$ is the $n$-th Bell number, that is the number of partitions of an $n$-element set into non-empty subsets. Clearly, $\xi_n$ can be thought of as the number of blocks in a uniform random partition of $\{1,2,\ldots,n\}$. Harper's CLT states that	
\begin{equation}\label{eq:CLT_stirl_sec}
\frac{\xi_n - \E \xi_n}{(\Var \xi_n)^{1/2}} \todistr \mathcal N(0,1),
\end{equation}
where $\mathcal N(0,1)$ is a random variable with the standard normal distribution, $\todistrnon$ denotes convergence in distribution,  and the normalization sequences are given explicitly by
$$
\E \xi_n = \frac{B_{n+1}}{B_n} - 1 \sim \frac n{\log n}, \qquad \Var \xi_n = \frac{B_{n+2}}{B_{n}} - \left(\frac{B_{n+1}}{B_{n}}\right)^2 -  1 \sim \frac{n}{\log^2 n},\quad n\to\infty.
$$
The notation $a(n) \sim b(n)$ means that $\lim_{n \to\infty} (a(n)/b(n)) = 1$.
For the proofs of these facts, see~\cite{harper} or~\cite[Chapter~4]{sachkov_book}. %By Slutsky's lemma, the variance in~\eqref{eq:CLT_stirl_sec} can be replaced by its asymptotic expression, but this does not apply to the expectation.

Harper's  proof that the zeros of $T_n$ are real and nonnegative is elementary and hinges on a simple recurrence formula connecting $T_n(x)$ and $T_{n-1}'(x)$. Much more sophisticated arguments based on saddle point asymptotics have been used by Elbert~\cite{elbert1,elbert2} to derive a limiting distribution for the empirical measure of zeros. More precisely, for a polynomial $p\in\C[x]$ denote by $\Zeros(p)$ the multiset of (complex) zeros of $p$ counted with multiplicities. Consider the following probability measure on the nonnegative half-line $[0,+\infty)$:
\begin{equation}\label{eq:elbert_measures}
\rho_n(\cdot):= \frac{1}{n} \sum_{x\in\Zeros(T_n)} \delta_{-x/n}(\cdot).
\end{equation}
Here, $\delta_x$ denotes the Dirac unit mass at $x\in\R$. Then, see Theorem 2.1 in~\cite{elbert2}, the sequence of probability measures $(\rho_n(\cdot))_{n\in\N}$ converges weakly to a probability measure $\rho(\cdot)$ concentrated on $[0,\eee]$ and having the Stieltjes transform
\begin{equation}\label{eq:elbert_limit}
\int_{[0,\eee]}\frac{\rho(\dd t)}{z-t}=1-\eee^{W_0(-1/z)},\quad z\in\mathbb{C}\setminus [0,\eee],
\end{equation}
where $W_0$ is the principal branch of the Lambert $W$-function. Furthermore, it is known that $\rho$ is absolutely continuous with a strictly decreasing density, see Theorem 2.3 in~\cite{elbert2}. In~\cite{elbert1} Elbert establishes strong asymptotics for the polynomials $x\mapsto T_n(nx)$ which, together with~\eqref{eq:elbert_measures}-\eqref{eq:elbert_limit}, implies that the smallest zero of $x\mapsto T_n(nx)$ converges to $-\eee$. This proves a conjecture stated at the bottom of p.~729 in~\cite{Mezo+Corcino:2015}.

Turning to the Stirling numbers of the first kind, let $\eta_n$ be a random variable with the following distribution:
\begin{equation}\label{eq:def_eta_n}
\P[\eta_{n}=k] = \frac 1 {n!} \stirling{n}{k} ,\qquad k\in \{1,2,\ldots,n\}.
\end{equation}
Then, Goncharov's CLT (see, e.g., \cite[\S 5.1.1]{sachkov_book}) states that
\begin{equation}\label{eq:CLT_stirl_first}
\frac{\eta_n - \E \eta_n}{(\Var \eta_n)^{1/2}} \todistr \mathcal N(0,1).
\end{equation}
The expectation and the variance of $\eta_n$ are given explicitly by
$$
\E \eta_n = \sum_{k=1}^n \frac 1k =\log n +O(1),
\qquad
\Var \eta_n = \sum_{k=1}^n \left(\frac 1k - \frac 1{k^2}\right) = \log n+O(1),\quad n\to\infty.
$$

It is remarkable that Harper's~\cite{harper} proof of~\eqref{eq:CLT_stirl_sec} uses only the fact that the zeros of the Touchard polynomial are nonpositive (without requiring  any information on their positions), together with the property $\Var \xi_n\to \infty$, as $n\to\infty$. Trivially, the zeros of the polynomials $S_n(x)$ are nonpositive, too, so that Harper's method proves Goncharov's CLT~\eqref{eq:CLT_stirl_first} as well. It is also obvious that empirical measures defined by~\eqref{eq:elbert_measures} with $\Zeros (T_n)$ replaced by $\Zeros (S_n)$ converges weakly to the Lebesgue measure on $[0,1]$. Many examples of probability generating functions having only nonpositive zeros, their connections with so-called Polya frequency sequences and a detailed bibliography can be found in~\cite{pitman_probab_bounds}; see also~\cite{brenti1988unimodal}.

As it is usual in probability theory, the central limit theorem comes together with further results such as the local limit theorem, the Edgeworth asymptotic expansion, large deviations (which may be precise or logarithmic), moderate deviations on various scales, normality zones,  and so on. The notion of mod-$\varphi$ convergence, introduced and developed by Nikeghbali and collaborators~\cite{barbour_kowalski_nikeghbali,delbaen_kowalski_nikeghbali,
feray_meliot_nikeghbali_book,jacod_kowalski_nikeghbali_mod_Gauss,kowalski_nikeghbali_mod_Poi,
kowalski_nikeghbali_zeta,meliot_nikeghbali_statmech} with an important early contribution by Hwang~\cite{hwang_LD1,hwang_convrates},  is a powerful tool which provides a unified approach to all these results. Once a suitable version of mod-$\varphi$ convergence has been verified, all these limit theorems follow automatically.
In the references cited above it has been demonstrated that mod-$\varphi$ convergence is a common phenomenon in probability theory, combinatorics, number theory and statistical mechanics; see the book~\cite{feray_meliot_nikeghbali_book} for an introduction to this subject.

Referring to Section~\ref{subsec:def_mod_phi} for the definition of mod-$\varphi$ convergence, we consider here one of its most basic examples which is provided by the sequence $(\eta_n)_{n\in\N}$ defined above. Using the Weierstrass product formula for the Gamma function it is easy to check~\cite[Example~2.1.3]{feray_meliot_nikeghbali_book}, see also~\cite{NZ13GeneralizedWeightedMeasure}, that
\begin{equation}\label{eq:stirling_first_mod_phi_classical}
\lim_{n\to\infty} \frac{\E \eee^{z \eta_n}}{\eee^{(\log n)(\eee^{z}-1)}} = \frac 1 {\Gamma(\eee^z)}
\end{equation}
locally uniformly in the complex variable $z\in \C$. Moreover, the speed of convergence in~\eqref{eq:stirling_first_mod_phi_classical} is $O(1/n)$, again locally uniformly in $z\in\C$.  The denominator on the left-hand side is the moment generating function of the Poisson distribution with parameter $\log n$. Therefore, the sequence $\eta_n$ is said to converge in the mod-Poisson sense with speed $\log n$.  Equation~\eqref{eq:stirling_first_mod_phi_classical} suggests the heuristic approximation
\begin{equation}\label{eq:mod_Poi}
\mlq\mlq \eta_n \eqdistr \text{Poi}(\log n) + \Xi +  O(1/n)\mrq\mrq,
\end{equation}
where $\eqdistr$ denotes equality in distribution, $\text{Poi}(\log n)$ is a random variable having Poisson distribution with parameter $\log n$, while $\Xi$ is an independent ``random variable'' with ``moment generating function'' $\E \eee^{z \Xi} = 1/\Gamma(\eee^z)$. Even though it turns out that there is no random variable $\Xi$ having the required moment generating function, \eqref{eq:mod_Poi} provides a very useful way of thinking about mod-$\varphi$ convergence. Indeed, all limit theorems for $\eta_n$ (including, for example, the complete asymptotic expansion of large deviation probabilities~\cite[\S~3.2]{feray_meliot_nikeghbali_book}) take the same form as they would do for the sequence of ``random variables'' $\text{Poi}(\log n) + \Xi$,
if $\Xi$ would exist.

The present paper started with an attempt to answer the question whether some analogues of~\eqref{eq:mod_Poi} hold for distributions related to the Stirling numbers of the \emph{second} kind. The purpose of our paper is two-fold. First, we prove that three natural families of probability distributions related to Stirling numbers of both kinds converge in the mod-$\varphi$ sense. Even in the case of Stirling numbers of the first kind, this result is new and is different from the known mod-Poisson convergence~\eqref{eq:stirling_first_mod_phi_classical}. As a consequence of mod-$\varphi$ convergence, we derive some limit theorems satisfied by these distributions. Second, we show how
the mod-$\varphi$ convergence is related to the limit theorems for empirical measures of zeros of the corresponding generating functions.

The paper is organized as follows. In Section~\ref{sec:families} we define three families of probability distributions, that are related to the Stirling numbers of both kinds, and that we shall work with.
Section~\ref{subsec:def_mod_phi} contains our main results, namely, the mod-$\varphi$ convergence of the three families of Stirling distributions introduced in Section~\ref{sec:families} and its various corollaries. In Section~\ref{sec:finite_free} we discuss limit theorems for the empirical measures of zeros of the corresponding generating functions. The proofs, which are skipped in the main part, are given in Sections~\ref{sec:mod-phi-proofs} and~\ref{sec:finite_free_proof}. In the Appendix we collect a number of analytic results used in our proofs dividing them into two parts. In the first part, that is Section~\ref{appendix_Lambert}, various properties of the Lambert $W$-function and its branches are gathered. In the second part (Section~\ref{sec:saddle_point}) we discuss how to derive uniform estimates in the classical saddle point method, which is our main tool in proving the mod-$\varphi$ convergence.

\section{Main results}

\subsection{Families of Stirling distributions}\label{sec:families}
We say that a random variable $X^{(1)}_{n,\theta}$ has a \emph{Stirling distribution of the first kind} (or a \emph{Stirling--Karamata distribution}) with parameters $n\in \N$ and $\theta>0$ if
\begin{equation}\label{eq:def_X_n_theta_no_n}
\P[X^{(1)}_{n,\theta}=k] = \frac{\stirling{n}{k} \theta^k}{S_n(\theta)},
\qquad
k\in \{1,2,\ldots,n\}.
\end{equation}
For $\theta=1$, this distribution appears as the law of the number of cycles $\eta_n$ in a uniform random permutation on $n$ elements, or as the law of the number of records in an i.i.d.\ sample of size $n$ from a continuous distribution. For general $\theta>0$, this is the distribution of the number of blocks in a random Ewens partition of $n$ elements with parameter $\theta$.

Similarly, we say that a random variable $X^{(2)}_{n,\theta}$ has the \emph{Stirling distribution of the second kind} with parameters  $n\in\N$ and $\theta>0$  if
\begin{equation}\label{eq:def_X_n_def}
\P[X^{(2)}_{n,\theta}=k] = \frac{\stirlingsec{n}{k} \theta^k}{T_n(\theta)},
\qquad
k\in \{1,2,\ldots,n\}.
\end{equation}
For $\theta=1$, this random variable counts the number of blocks $\xi_n$ in a uniform random partition of the set $\{1,\ldots,n\}$ and admits an elegant probabilistic representation due to Stam~\cite{stam83}. For general $\theta>0$, it describes the number of blocks in a  Gibbs random partition~\cite[Section~1.5]{pitman_book}.

Another probability distribution involving Stirling numbers of the second kind can be constructed using the identity
$$
x^n=\sum_{k=1}^{n}\stirlingsec{n}{k}x(x-1)\cdots(x-k+1),\quad n\in\N.
$$
Let $\theta$ be either a positive integer or a real number such that $\theta>n-1$. We say that a random variable $X^{(3)}_{n,\theta}$ has the {\em Stirling-Sibuya distribution} if
$$
\P[X^{(3)}_{n,\theta}=k] = \stirlingsec{n}{k} \frac{(\theta)^{\underline{k}}}{\theta^n} = \stirlingsec{n}{k} \frac{\theta(\theta-1)\cdots(\theta-k+1)}{\theta^n},
\qquad
k\in \{1,2,\ldots,n\},
$$
see \cite{Sibuya:2006}. Note that if $\theta\in \{1,2,\ldots, n\}$, then $X^{(3)}_{n,\theta}$ is  concentrated on  $\{1,2,\ldots,\theta\}$. If $\theta$ is a real parameter, then the constraint $\theta>n-1$ is needed to ensure positivity of the weights. Furthermore, if $\theta$ is an integer, then $X^{(3)}_{n,\theta}$ counts the number of occupied boxes when $n$ (distinguishable) balls are allocated equiprobably and independently among $\theta$ boxes. Alternatively, $X^{(3)}_{n,\theta}$ is the number of distinct values in a sample of $n$ i.i.d. random variables with the uniform distribution on $\{1,2,\ldots,\theta\}$. Distributions related to Stirling numbers have been reviewed in~\cite{charalambides,Sibuya:2006}.

For future use let us introduce the following notation for the generating functions of $X_{n,\theta}^{(i)}$, $i=1,2,3$. Put
$$
\mathcal{G}_{n,\theta}^{(i)}(t):=\sum_{k=1}^{n}\P[X^{(i)}_{n,\theta}=k] t^k,\quad t\in\C,\quad i=1,2,3.
$$
It is clear that $\mathcal{G}_{n,\theta}^{(i)}$ is a polynomial of degree $n$ if $i=1,2$. Furthermore,
$$
\mathcal{G}_{n,\theta}^{(1)}(t)~=~
\frac{S_n(\theta t)}{S_n(\theta)},\quad
\mathcal{G}_{n,\theta}^{(2)}(t)~=~
\frac{T_n(\theta t)}{T_n(\theta)},\quad t\in\C.
$$
The polynomial $\mathcal{G}_{n,\theta}^{(3)}$ is given by
$$
\mathcal{G}_{n,\theta}^{(3)}(t)~=~\sum_{k=1}^{n}\stirlingsec{n}{k} \frac{\theta(\theta-1)\cdots(\theta-k+1)}{\theta^n}t^k,\quad t\in\C.
$$
and does not seem to belong to any simple family. The degree of $\mathcal{G}_{n,\theta}^{(3)}$ is equal to $n$ if $\theta>n-1$, and is equal to $\theta$ if $\theta\leq n$ is an integer. Note also that $\mathcal{G}_{n,\theta}^{(3)}$ is well-defined for non-integer $\theta\leq n-1$, yet it is not a generating function of any probability law. An integral representation for $\mathcal{G}_{n,\theta}^{(3)}$ will be given in Lemma~\ref{lem:stirling_3_integral_rep}. We shall also show that $\mathcal{G}_{n,\theta}^{(3)}$ is the so-called finite free multiplicative convolution of the Touchard and generalized Laguerre polynomials, see Section~\ref{sec:finite_free}.

In this paper we shall be interested in the asymptotics of the above distributions with the tilted parameter $\theta=n\vartheta$, where $\vartheta>0$ is fixed. %The mod-$\varphi$ behaviour under the Ewens measure for fixed $\theta$ was studied in .

%The moment generating functions of $X^{(1)}_{n,\vartheta n}$ and $X^{(2)}_{n,\vartheta n}$ are
%\begin{align}
%\E \eee^{zX^{(1)}_{n,\vartheta n}}
%&=\mathcal{G}_{n,\vartheta n}^{(1)}(\eee^z)=
%\frac{S_n(\vartheta \eee^z n)}{S_n(\vartheta n)}, \quad z\in\C,%\label{eq:moment_gener_X_n_1}\\
%\E \eee^{z X^{(2)}_{n,\vartheta n }}
%&=\mathcal{G}_{n,\vartheta n}^{(2)}(\eee^z)
%=\frac{T_n(\vartheta \eee^z n)}{T_n(\vartheta n)},\quad z\in\C. %\label{eq:moment_gener_X_n}
%\end{align}

\subsection{Mod-\texorpdfstring{$\varphi$}{phi} convergence for the families of Stirling distributions}\label{subsec:def_mod_phi}

\subsubsection{Definition of mod-\texorpdfstring{$\varphi$}{phi} convergence.}
In the literature several non-equivalent definitions of mod-$\varphi$ convergence appear. The notion we shall use here is close but not equivalent to the definition used in the book~\cite[Definition~1.1.1]{feray_meliot_nikeghbali_book}.
Let $(X_n)_{n\in\N}$ be a sequence of random variables with values in $\R$ whose Laplace transforms $\E \eee^{z X_n}$ exist in some  strip $\mathcal H=\{z\in \C: \Re z \in (\beta_-,\beta_+)\}$ with $-\infty \leq \beta_- < \beta_+ \leq +\infty$.
In order to state the definition of mod-$\varphi$ convergence we need the following ingredients:
\begin{itemize}
\item[($\Phi$1)] a sequence $(w_n)_{n\in\N}$ of positive numbers with $\lim_{n\to\infty} w_n = +\infty$;
\item[($\Phi$2)] an open, connected set $\mathcal D \subset \mathcal H$ containing  the interval $(\beta_-,\beta_+)$;
\item[($\Phi$3)] an analytic function $\varphi:\mathcal D \to\C$ whose restriction to the interval $(\beta_-,\beta_+)$ is real-valued and strictly convex;
\item[($\Phi$4)] an analytic function $\Psi(z) : \mathcal D \to \C$ which does not vanish on $\mathcal{D}\cap \R$.
\end{itemize}

The sequence $(X_n)_{n\in\N}$ is said to \emph{converge mod-$\varphi$} with parameters listed in ($\Phi$1)-($\Phi$4) if
\begin{equation}\label{eq:mod_phi_def_conv}
\lim_{n\to\infty} \frac{\E\eee^{zX_n}}{\eee^{w_n\varphi(z)}} = \Psi(z),
\end{equation}
and the convergence is uniform on compact subsets of  $\mathcal D$. In~\cite[Definition~1.1.1]{feray_meliot_nikeghbali_book}, $\eee^{\varphi(z)}$ is required to be a moment generating function of some infinitely divisible distribution, although the most essential limit theorems of~\cite{feray_meliot_nikeghbali_book} continue to hold without this requirement, as discussed in~\cite[\S~4.5.2]{feray_meliot_nikeghbali_book}, see also~\cite{kabluchko_marynych_sulzbach}. We omit this requirement since in our examples $\eee^{\varphi(z)}$  cannot be represented as a Laplace transform of a probability distribution. We refer to~\cite[\S~4.5.3]{feray_meliot_nikeghbali_book} for a discussion of the closely related notion of quasi-powers introduced by Hwang~\cite{hwang_LD1}, \cite{hwang_convrates}.

\subsubsection{Mod-\texorpdfstring{$\varphi$}{phi} convergence for \texorpdfstring{$X_{n,\vartheta n}^{(i)},i=1,2,3$}{X}.}

To state our results on mod-$\varphi$ convergence we shall need the following functions. Put
\begin{equation}\label{eq:def_L_1}
L_1(z) := (z+1)\log (z+1) - z \log z,\quad z \in\C \setminus (-\infty,0],
\end{equation}
and
\begin{equation}\label{eq:def_L_2}
L_2(z) :=  W_0(1/z) + \frac{1}{W_0(1/z)} - z + \log z,\quad z\in\C \setminus (-\infty,0].
\end{equation}
Recall that $W_0$ denotes the principal branch of the Lambert $W$-function. The definition and some useful properties of this function are collected in Appendix~\ref{appendix_Lambert}. In particular, since $W_0$ is analytic in the slitted plane $\C\setminus (-\infty,-1/\eee]$ and does not vanish, $L_2$ is well-defined and analytic in $\C\setminus (-\infty,0]$.

\begin{theorem}\label{thm:mod_phi_stirling1}
Fix $\vartheta>0$. Then, there is an open set $\mathcal{D}_1\subset \C$ such that $\R \subset \mathcal{D}_1$ and such that the moment generating functions of the random variables $X^{(1)}_{n,\vartheta n}$ satisfy
\begin{equation}\label{eq:mod_phi_1}
\Psi_1(z)
:=
\lim_{n\to\infty} \frac{\E \eee^{z X^{(1)}_{n,\vartheta n}}}{\eee^{n \varphi_1(z;\vartheta)}}
=
\eee^{z/2} \sqrt{\frac{\vartheta+1}{\vartheta \eee^z +1}},
\qquad z\in \mathcal{D}_1,
\end{equation}
where $\varphi_1(z;\vartheta) := L_1(\vartheta \eee^z) - L_1(\vartheta)$ with $L_1$ as in~\eqref{eq:def_L_1}. Moreover, \eqref{eq:mod_phi_1} holds locally uniformly on $\mathcal{D}_1$ with the speed $O(1/n)$.
\end{theorem}

\begin{theorem}\label{thm:mod_phi_stirling2}
Fix $\vartheta>0$. Then, there is an open set $\mathcal{D}_2\subset \C$ such that $\R \subset \mathcal{D}_2$ and such that the moment generating functions of the random variables $X^{(2)}_{n,\vartheta n}$ satisfy
\begin{equation}\label{eq:mod_phi_2}
\Psi_2(z) :=
\lim_{n\to\infty} \frac{\E \eee^{zX^{(2)}_{n,\vartheta n}}}{\eee^{n \varphi_2(z;\vartheta)}}
=
\sqrt{\frac{W_0(\vartheta^{-1})+1}{W_0(\vartheta^{-1} \eee^{-z})+1}},
\qquad z\in \mathcal{D}_2,
\end{equation}
where $\varphi_2(z;\vartheta) := L_2(\vartheta \eee^z) - L_2(\vartheta)$ with $L_2$ as in~\eqref{eq:def_L_2}. Moreover, \eqref{eq:mod_phi_2} holds locally uniformly on $\mathcal{D}_2$ with the speed $O(1/n)$.
\end{theorem}

For the Stirling-Sibuya distribution $X^{(3)}_{n,\vartheta n}$ the mod-$\varphi$ convergence looks as follows.

\begin{theorem}\label{thm:mod_phi_stirling3}
Fix $\vartheta>0$. Then, there is an open set $\mathcal D_3\subset \C$ such that $\R \subset \mathcal{D}_3$ and such that the moment generating functions of the random variable $X^{(3)}_{n,\vartheta n}$ satisfy
\begin{equation}\label{eq:mod_phi_3}
\lim_{n\to\infty}\frac{\E \eee^{z X^{(3)}_{n,\vartheta n}}}{\eee^{n\varphi_3(z;\vartheta)}}=\sqrt{\frac{\vartheta}{\vartheta L_3(z;\vartheta)+\vartheta-1}},\quad z\in\mathcal{D}_3,
\end{equation}
where
$$
L_3(z;\vartheta):=1/\vartheta+W_0(\vartheta^{-1}(\eee^{-z}-1)\eee^{-1/\vartheta}),\quad z\in \mathcal{D}_3,
$$
and
$$
\varphi_3(z;\vartheta):=(\vartheta-1)\log\vartheta-1+(\vartheta-1)\log L_3(z;\vartheta)+\vartheta(z+L_3(z;\vartheta)).
$$
Moreover, \eqref{eq:mod_phi_3} holds locally uniformly on $\mathcal D_3$ with speed $O(1/n)$.
\end{theorem}
\begin{remark}
Note that the statement of Theorem~\ref{thm:mod_phi_stirling3} holds also for $0< \vartheta <1$ without the assumption that $\vartheta n$ is integer. If $\vartheta<1$ and $n\vartheta$ is not an integer, then the numerator in the left-hand side should be interpreted as $\mathcal{G}_{n,\vartheta n}^{(3)}(\eee^z)$ which is well-defined for all $\vartheta>0$ and $n\in\N$. However, in this case~\eqref{eq:mod_phi_3} does not have a clear probabilistic meaning.
\end{remark}

For $\vartheta=1$ there are significant simplifications in Theorem~\ref{thm:mod_phi_stirling3}. Recall that $X_{n,n}^{(3)}$ is the number of distinct values in a sample of $n$ i.i.d. random variables with the uniform distribution on $\{1,2,\ldots,n\}$.

\begin{corollary}
There is an open subset of $\C$ which contains $\R$ and such that locally uniformly on this set
\begin{equation}\label{eq:mod_phi_3_theta=1}
\lim_{n\to\infty}\frac{\E \eee^{z X^{(3)}_{n,n}}}{\eee^{n\varphi_3(z;1)}}=\sqrt{\frac{1}{1+W_0(\eee^{-1}(\eee^{-z}-1))}},
\end{equation}
with the speed of convergence $O(1/n)$, where
$$
\varphi_3(z;1)=z+W_0(\eee^{-1}(\eee^{-z}-1)).
$$
\end{corollary}

The proofs of the above theorems will be given in Section~\ref{sec:mod-phi-proofs}.

\subsubsection{Corollaries of the mod-$\varphi$ convergence.}

The mod-$\varphi$ convergence obtained above can be used to derive probabilistic limit theorems on the Stirling distributions of the aforementioned three kinds. For simplicity we shall formulate our results in case of $X_{n,\vartheta n}^{(3)}$ only for $\vartheta\geq 1$, in order to avoid problems when $n\vartheta$ is not an integer and $\vartheta\in (0,1)$. The latter case will be discussed in Remark~\ref{rem:x_3_theta<1} at the end of this subsection.

To state these theorems we need to introduce additional notation. First of all, let us record the formulas for the first two derivatives (with respect to $z$ while $\vartheta>0$ stays fixed)
\begin{itemize}
\item of the function $\varphi_1(z;\vartheta)$:
\begin{align*}
\varphi_1'(z;\vartheta)
&=
\vartheta \eee^z L_1'(\vartheta \eee^z)
=
\vartheta \eee^z \log (1+\vartheta^{-1}\eee^{-z}), \\
\varphi_1''(z;\vartheta)
&=
\vartheta \eee^z L_1'(\vartheta \eee^z) +  (\vartheta \eee^z)^2 L_1''(\vartheta \eee^z)
=
\vartheta \eee^z \log (1+\vartheta^{-1}\eee^{-z}) - \frac 1 {1+\vartheta^{-1}\eee^{-z}};
\end{align*}
\item of the function $\varphi_2(z;\vartheta)$:
\begin{align*}
\varphi_2'(z;\vartheta)
&=
\vartheta \eee^z L_2'(\vartheta \eee^z)
=
\frac 1 {W_0(\vartheta^{-1}\eee^{-z})} - \vartheta \eee^z,\\
\varphi_2''(z;\vartheta)
&
=
\vartheta \eee^z L_2'(\vartheta \eee^z) +  (\vartheta \eee^z)^2 L_2''(\vartheta \eee^z)
=\frac 1 {W_0(\vartheta^{-1}\eee^{-z}) (1+W_0(\vartheta^{-1}\eee^{-z}))} - \vartheta \eee^z;
\end{align*}
\item of the function $\varphi_3(z;\vartheta)$:
\begin{align*}
\varphi_3'(z;\vartheta)
&=\vartheta+\frac{\vartheta^2 W_0(\vartheta^{-1}(\eee^{-z}-1)\eee^{-1/\vartheta})}{(\eee^z-1)(1+\vartheta W_0(\vartheta^{-1}(\eee^{-z}-1)\eee^{-1/\vartheta}))},\\
\varphi_3''(z;\vartheta)
&
=\frac{\vartheta^2 W_0(\vartheta^{-1}(\eee^{-z}-1)\eee^{-1/\vartheta})}{(\eee^z-1)^2 (1+\vartheta W_0(\vartheta^{-1}(\eee^{-z}-1)\eee^{-1/\vartheta}))}\\
&\times\left(\frac{1}{(1+\vartheta W_0(\vartheta^{-1}(\eee^{-z}-1)\eee^{-1/\vartheta}))(1+W_0(\vartheta^{-1}(\eee^{-z}-1)\eee^{-1/\vartheta}))}-\eee^z\right).
\end{align*}
\end{itemize}

The last four formulas follow  after some straightforward transformations invoking relation~\eqref{eq:W_0_der} from the Appendix. The following functions $\mu_i(\vartheta)$ and $\sigma_i^2(\vartheta)$ play a role of the asymptotic expectation and variance of the Stirling distribution of type $i\in \{1,2,3\}$ with parameters $(n,\vartheta n)$:
\begin{align*}
\mu_1(\vartheta)
&:=
\varphi_1'(0;\vartheta)
=
\vartheta \log \left(1+\frac 1 \vartheta\right),
\quad
\sigma^2_1(\vartheta):=
\varphi_1''(0;\vartheta)
=
\vartheta \log \left(1+\frac 1 \vartheta\right) - \frac \vartheta{1+\vartheta},\\
\mu_2(\vartheta)
&:=
\varphi_2'(0;\vartheta)
=
\frac 1 {W_0(1/\vartheta)} - \vartheta,\quad
\sigma_2^2(\vartheta):= \varphi_2''(0;\vartheta) = \frac{1}{W_0(\vartheta^{-1})(1+W_0(\vartheta^{-1}))} - \vartheta,\\
\mu_3(\vartheta)
&:=
\varphi_3'(0;\vartheta)
=
\vartheta(1-\eee^{-1/\vartheta}),\quad
\sigma_3^2(\vartheta)
:= \varphi_3''(0;\vartheta) = \vartheta \eee^{-1/\vartheta}\left(1-\eee^{-1/\vartheta}-\frac{1}{\vartheta}\eee^{-1/\vartheta}\right),
\end{align*}
where the last two equations follow from the asymptotic relation $W_0(z)\sim z$, as $z\to 0$.
%$$
%W_0(z)\sim z,\quad z\to 0.
%$$
These functions are  defined for $\vartheta>0$. The graphs of $\mu_i(\vartheta)$ and $\sigma_i(\vartheta)$ are shown on Figure~\ref{fig:function_mu_sigma}. The main properties of these functions are collected in the following lemma.

%%%%%%%%%%%%%%%%%%For figures see "Touchard Lambert Simulations.nb"
\begin{figure}[t]
\begin{center}
\includegraphics[width=0.33\textwidth]{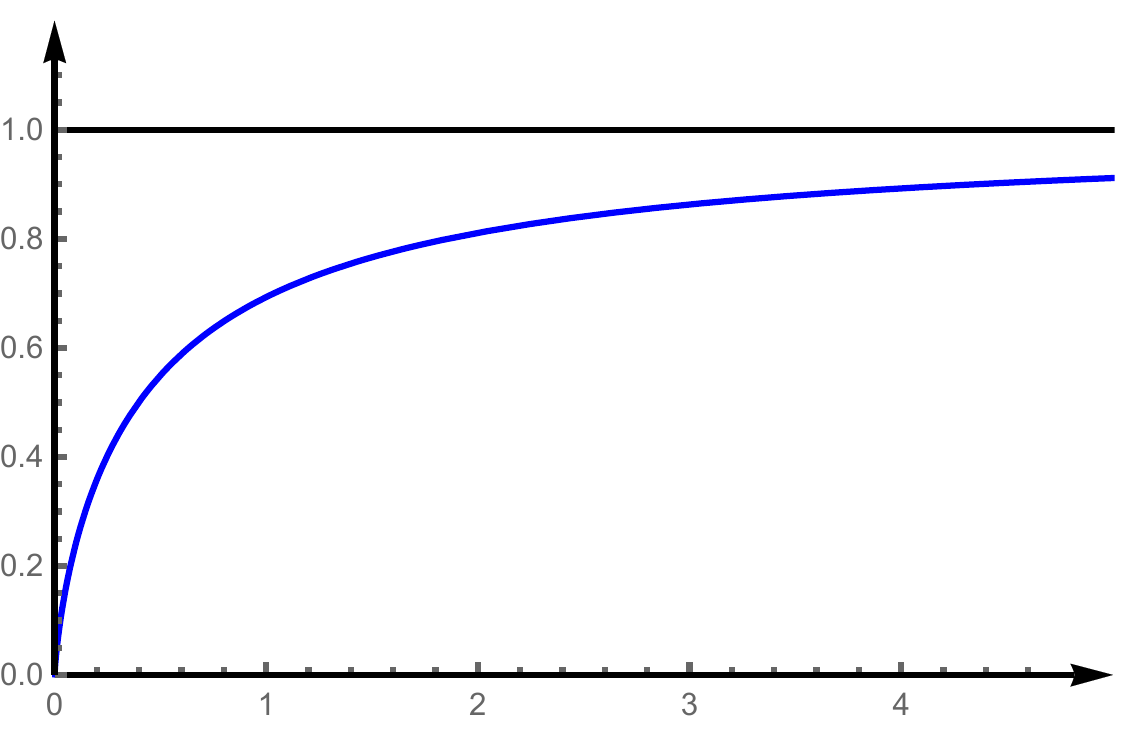}
\includegraphics[width=0.33\textwidth]{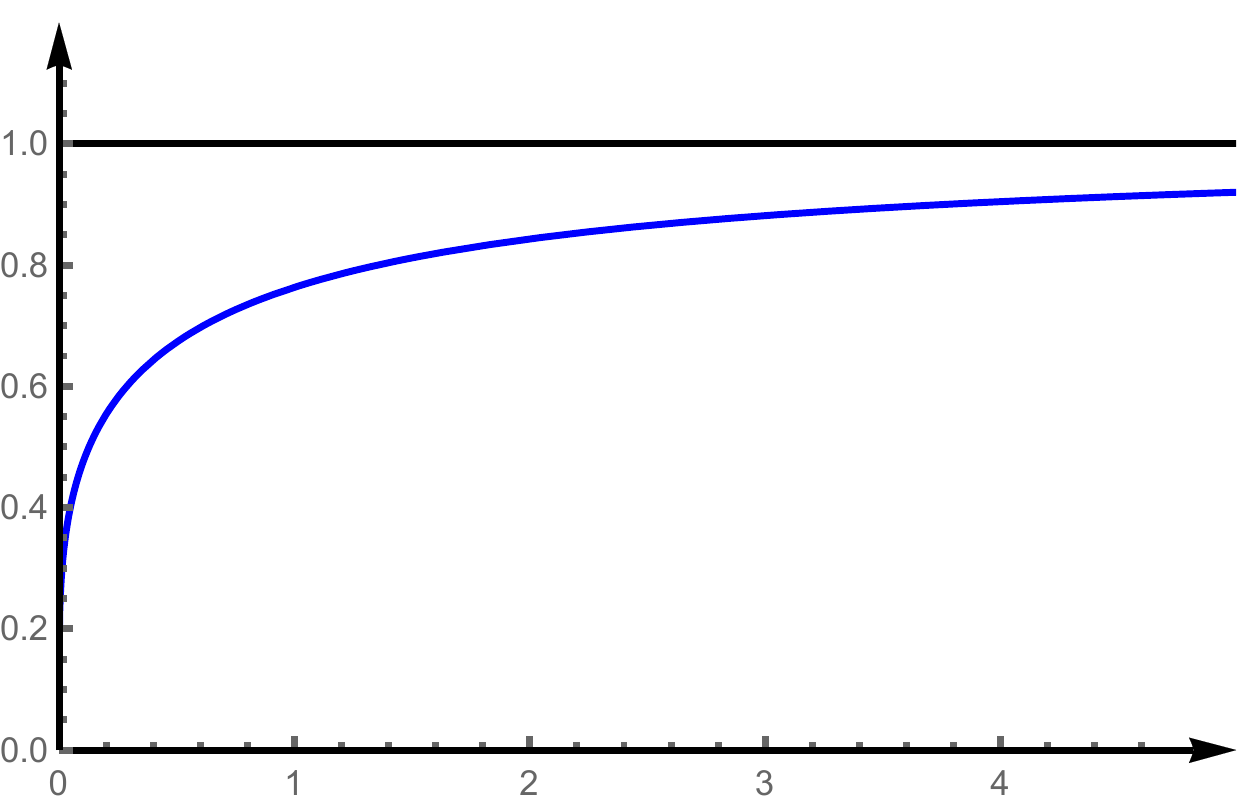}
\includegraphics[width=0.33\textwidth]{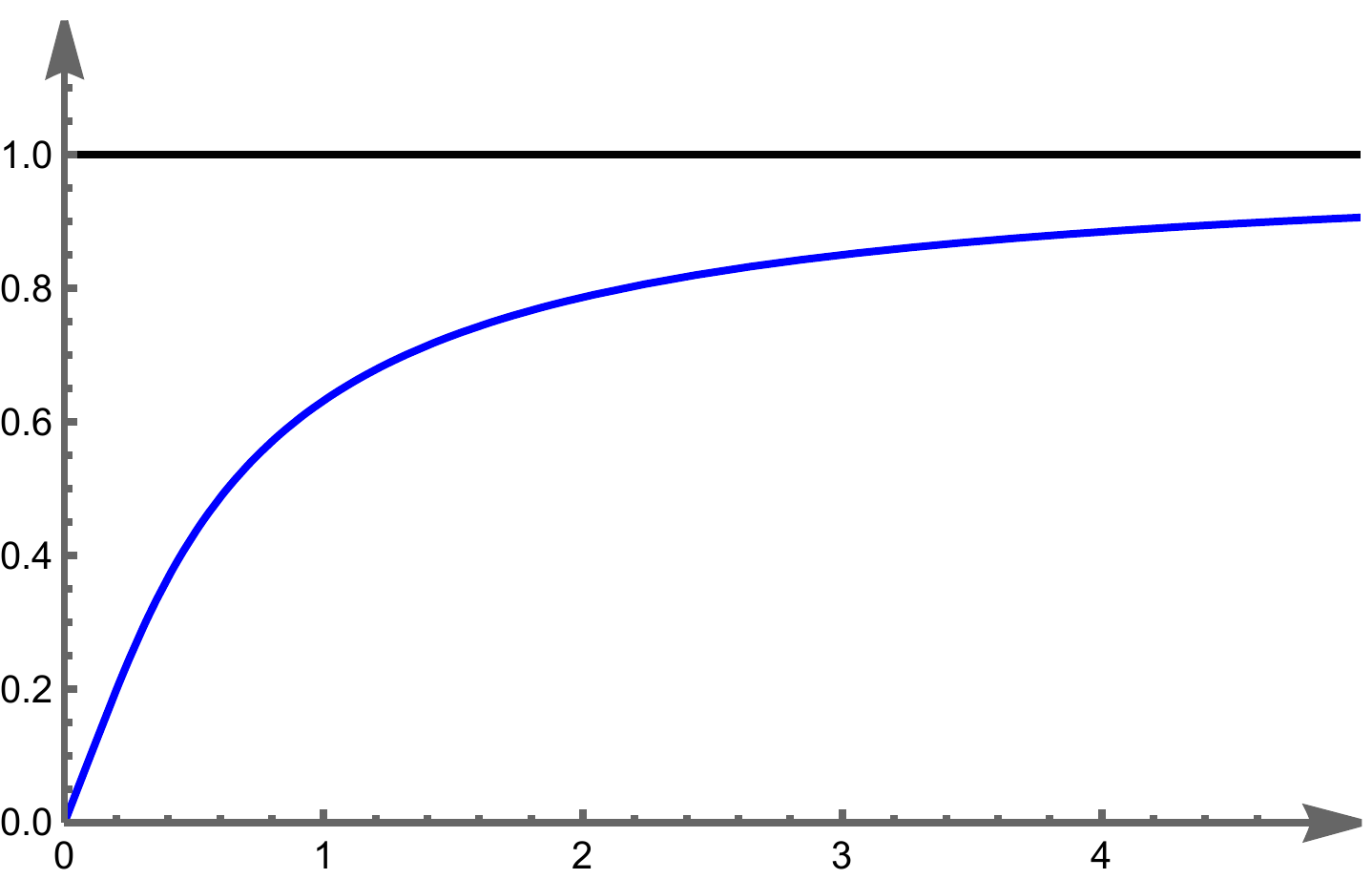}
\includegraphics[width=0.33\textwidth]{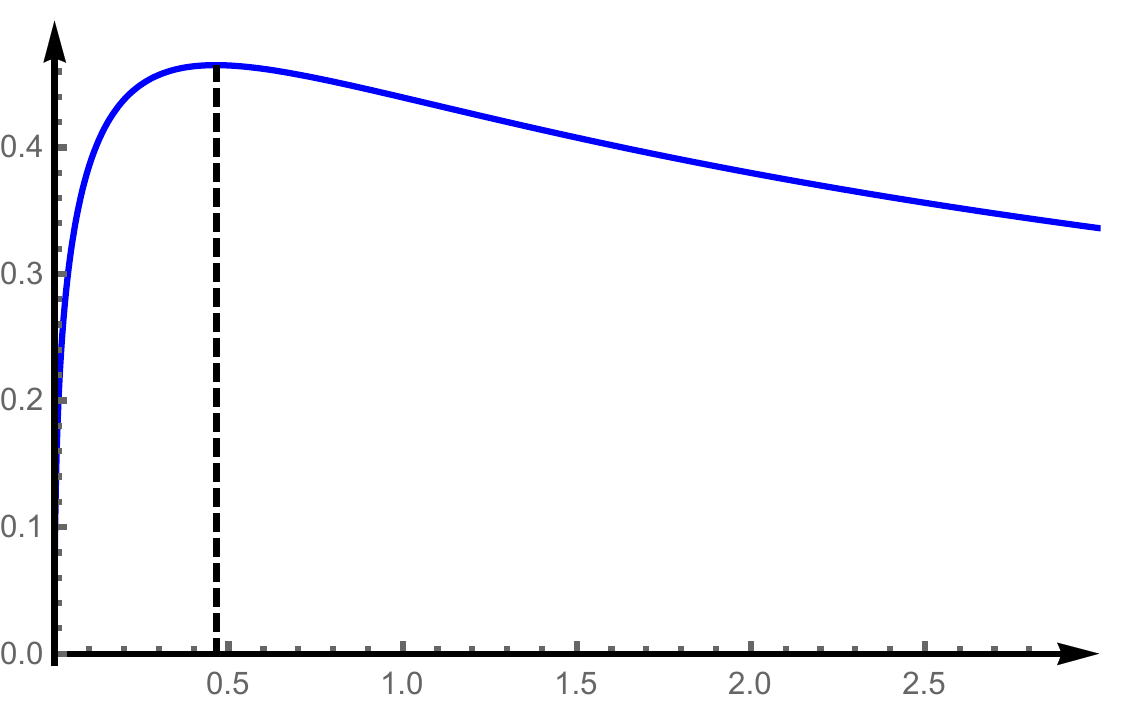}
\includegraphics[width=0.33\textwidth]{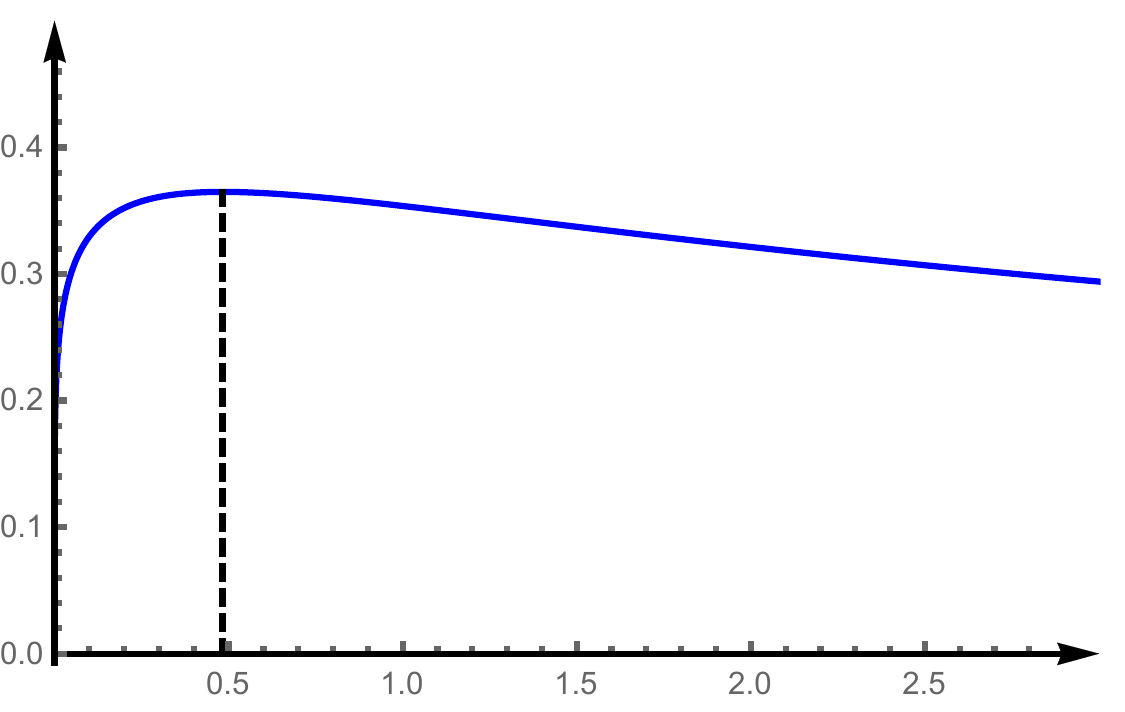}
\includegraphics[width=0.33\textwidth]{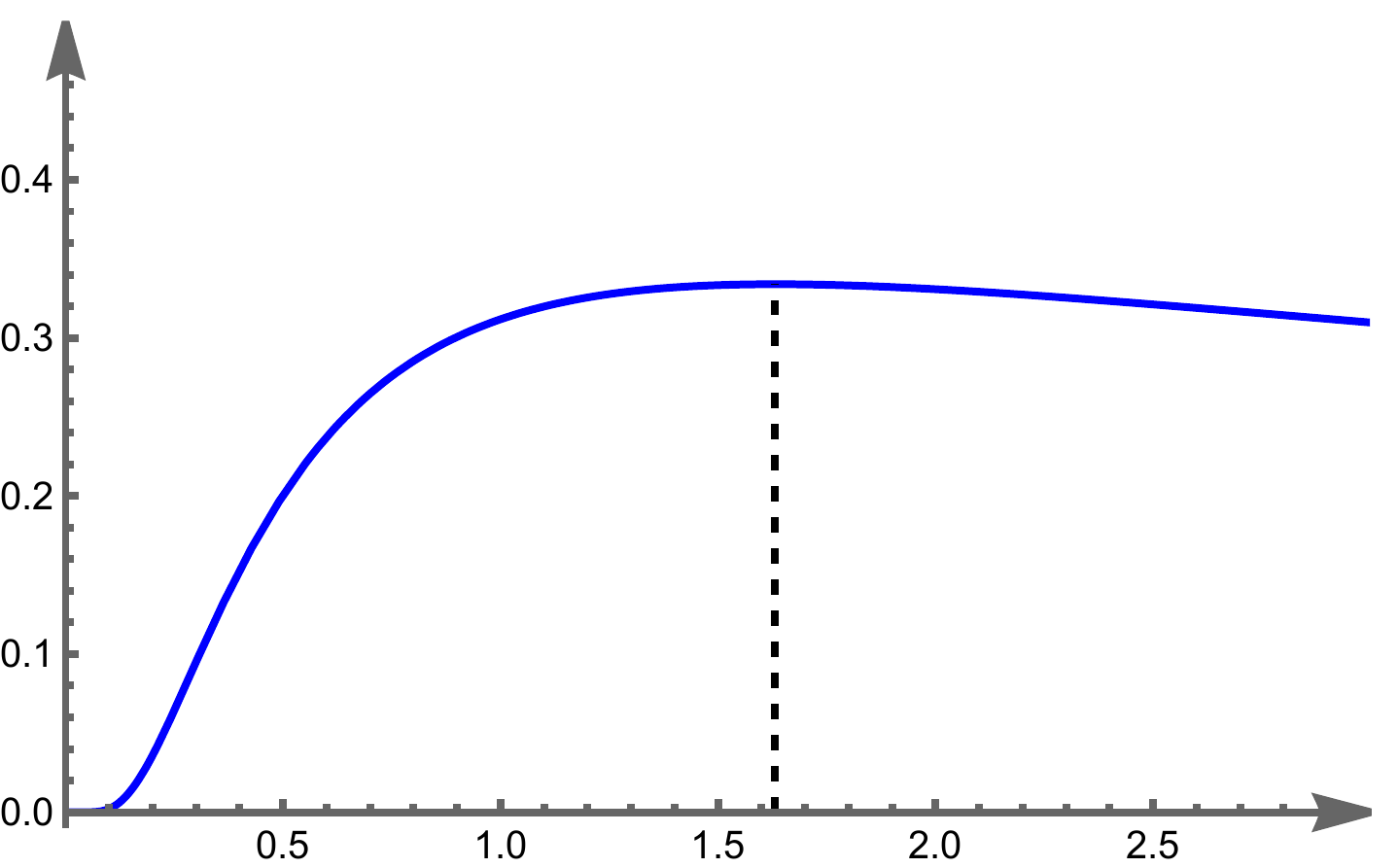}
\end{center}
\caption
{
Top row: The graph of the function $\mu_i(\vartheta)$, $\vartheta>0$.  Bottom row: The graph of the function $\sigma_i(\vartheta)$, $\vartheta>0$. Left column: $i=1$. Middle column: $i=2$. Right column: $i=3$. The dashed vertical line shows the maximum of $\sigma_i(\vartheta)$ which is attained at $\vartheta = \vartheta_i^*$ with $\vartheta_1^* \approx 0.46241$, $\vartheta_2^* \approx 0.48273$ and $\vartheta_2^* \approx 1.6313$.
}
\label{fig:function_mu_sigma}
\end{figure}
\begin{lemma}\label{lem_mu_sigma_properties}
Let $i\in \{1,2,3\}$. The function  $\mu_i(\vartheta)$ is smooth and strictly monotone increasing on $(0,\infty)$. Also,
$$
\lim_{\vartheta\downarrow 0} \mu_i(\vartheta) =  0,
\qquad
\lim_{\vartheta\uparrow +\infty} \mu_i(\vartheta) = 1,
\qquad
\lim_{\vartheta\uparrow +\infty} \sigma_i^2(\vartheta) = \lim_{\vartheta\downarrow 0} \sigma_i^2(\vartheta) = 0,
$$
and  $\sigma_i^2(\vartheta)>0$ for all $\vartheta>0$.
\end{lemma}
As a consequence, observe that the functions $\varphi_i(z;\vartheta)$ are strictly convex in $z\in\R$ for $i=1,2$.  Indeed, $\varphi_i'(z;\vartheta) = \mu_i(\vartheta \eee^z)$ is strictly increasing.

The next lemma states that $\mu_i(\vartheta)$ and $\sigma^2_i(\vartheta)$ are the linear growth rates of $\E X_{n,\vartheta n}^{(i)}$ and $\Var X_{n,\vartheta n}^{(i)}$, respectively. This result follows immediately from Theorems~\ref{thm:mod_phi_stirling1}, \ref{thm:mod_phi_stirling2} and \ref{thm:mod_phi_stirling3} since the uniform convergence of analytic functions in a complex neighborhood of the origin implies convergence of all their derivatives evaluated at zero.
\begin{lemma}\label{lem:linear_growth}
For $i\in \{1,2\}$ and all $\vartheta >0$, and for $i=3$ and $\vartheta\geq 1$, we have
$$
\lim_{n\to\infty} \frac 1n \E X_{n,\vartheta n}^{(i)} = \mu_i(\vartheta),
\qquad
\lim_{n\to\infty} \frac 1n \Var X_{n,\vartheta n}^{(i)} = \sigma^2_i(\vartheta).
$$
\end{lemma}

The next result, which is local limit theorems, can be obtained by an appeal to Theorem 2.7 in~\cite{kabluchko_marynych_sulzbach} applied to deterministic profiles (we use the terminology of~\cite{kabluchko_marynych_sulzbach}) $\mathbb{L}_n^{(i)}(k):=\P[X_{n,\vartheta n}^{(i)}=k]$, $i=1,2,3$. The assumptions (A1) and (A2) of the cited paper hold with $\beta_{-}=-\infty$, $\beta_{+}=\infty$, $w_n=n$, and (A2) is a consequence of the mod-$\varphi$ convergence stated in Theorems~\ref{thm:mod_phi_stirling1},
~\ref{thm:mod_phi_stirling2} and~\ref{thm:mod_phi_stirling3}. The assumption (A3) holds with $r=1$, since the speed of convergence is $O(1/n)=O(1/w_n)$. The assumption (A4) will  be checked below in Section~\ref{sec:local_limit}.

\begin{theorem}\label{theo:local_limit2}
Let $\vartheta>0$. The following hold true:
\begin{align}
&\lim_{n\to\infty}
\sqrt n \sup_{k\in \{0,\ldots,n\}} \left|\frac{\stirling{n}{k} (\vartheta n)^k}{S_n(\vartheta n)}
- \frac{1}{\sqrt{2\pi n} \, \sigma_1(\vartheta)} \exp\left\{-\frac {(k - \mu_1(\vartheta)n)^2}{2\sigma_1^2(\vartheta) n}\right\} \right| = 0,\label{eq:theo:local_limit2_eq1}\\
&\lim_{n\to\infty}
\sqrt n \sup_{k\in \{0,\ldots,n\}} \left|\frac{\stirlingsec{n}{k} (\vartheta n)^k}{T_n(\vartheta n)}
- \frac{1}{\sqrt{2\pi n} \, \sigma_2(\vartheta)} \exp\left\{-\frac {(k - \mu_2(\vartheta)n)^2}{2\sigma_2^2(\vartheta) n}\right\} \right| = 0,\label{eq:theo:local_limit2_eq2}\\
&\lim_{n\to\infty}
\sqrt n \sup_{k\in \{0,\ldots,n\}} \left|\stirlingsec{n}{k}\frac{(\vartheta n)^{\underline{k}}}{(\vartheta n)^n}
- \frac{1}{\sqrt{2\pi n} \, \sigma_3(\vartheta)} \exp\left\{-\frac {(k - \mu_3(\vartheta)n)^2}{2\sigma_3^2(\vartheta) n}\right\} \right| = 0.\label{eq:theo:local_limit2_eq3}
\end{align}
Moreover, the convergence is uniform in $\vartheta$ as long as $\vartheta$ stays in any compact subset of $(0,\infty)$.
\end{theorem}

\begin{remark}
Recall that $X_{n,\theta}^{(3)}$ has the same distribution as the number of occupied boxes when $n$ balls are allocated equiprobably  and independently among $\theta$ boxes. The local limit theorem for the latter when $\theta=\theta_n\sim \vartheta \cdot n$ is known, see Theorem 1 on p.~54 in~\cite{Kolchin_book}.
\end{remark}

\begin{figure}[t]
\begin{center}
\includegraphics[width=0.45\textwidth]{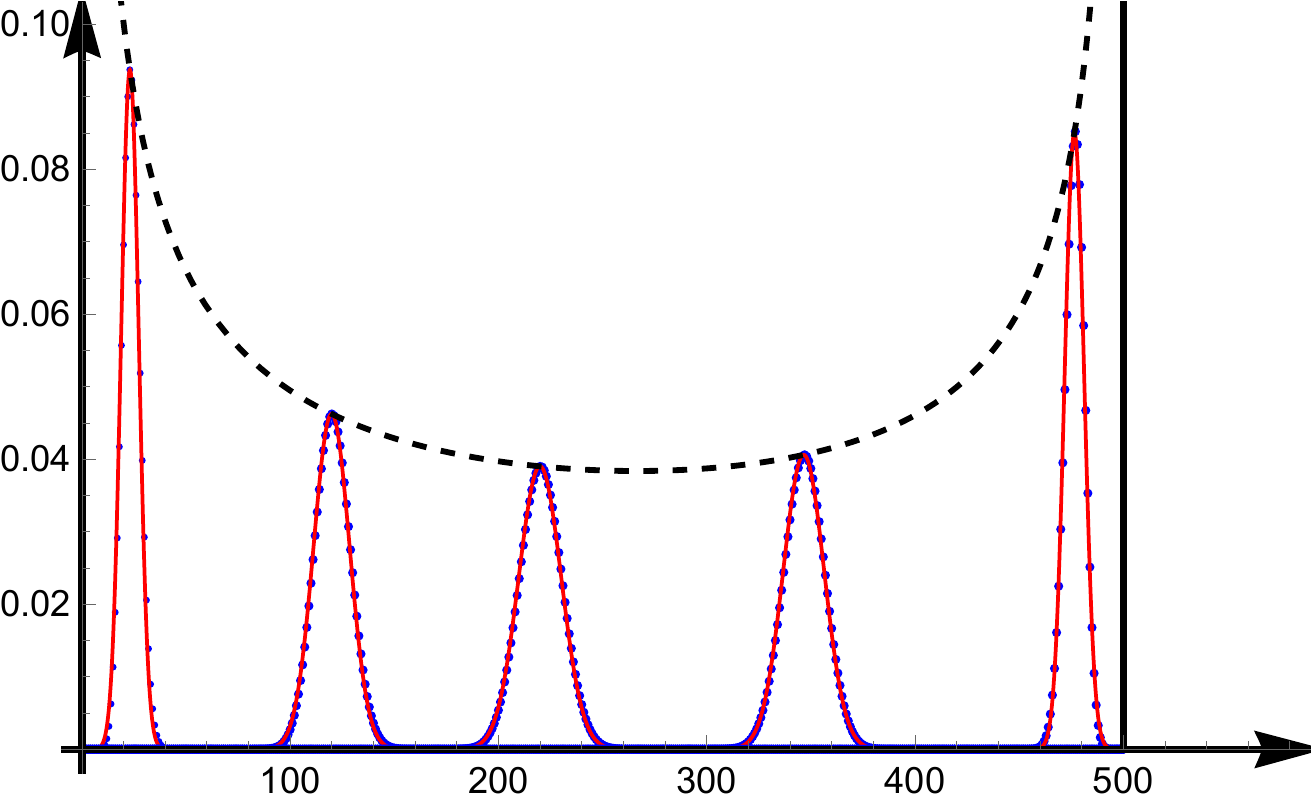}
\includegraphics[width=0.45\textwidth]{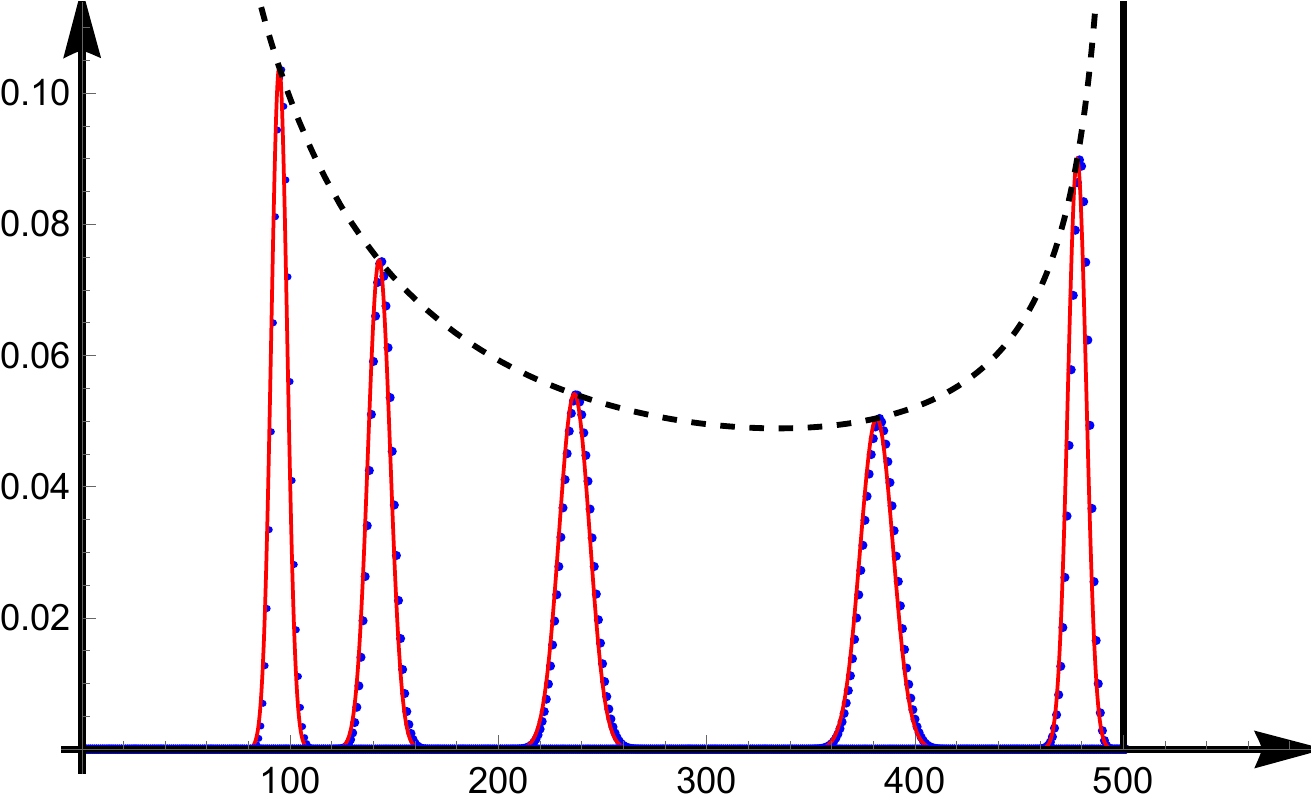}
\includegraphics[width=0.45\textwidth]{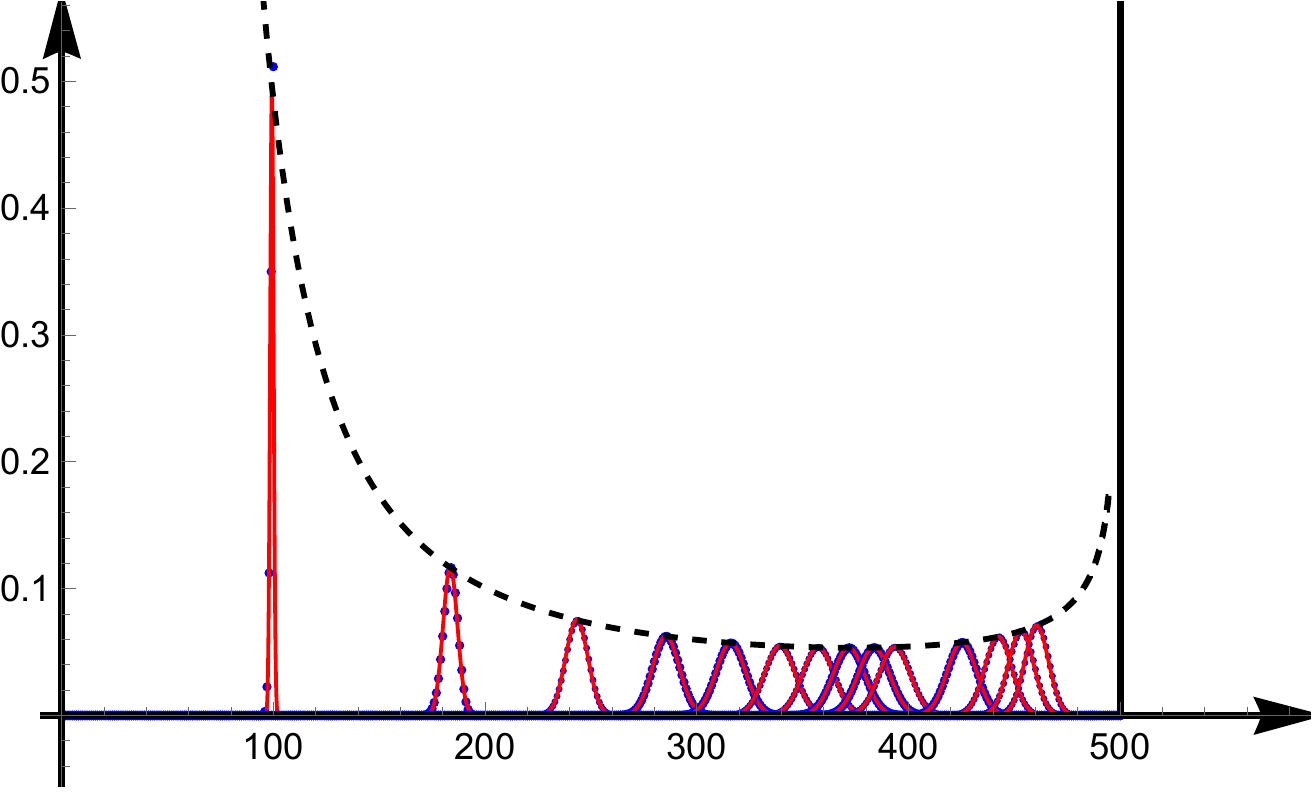}
\end{center}
\caption
{
Local limit theorems for Stirling distributions with parameters $(n,\vartheta n)$ and $n=500$; see Theorem~\ref{theo:local_limit2}. Top left: Stirling distributions of the first kind with $\vartheta\in \{0.01,0.1,0.3,1,10\}$. Top right: Stirling distribution of the second kind with $\vartheta\in \{0.001, 0.01,0.1,1,10\}$. Bottom: Stirling distribution of the third kind with $\vartheta\in \{0.2, 0.4,\ldots,1.8,2,3,4,5,6\}$. Blue disks show the probability mass functions of the corresponding distributions. Red bell-shaped curves are the local limit theorem approximations to these distributions as given by Theorem~\ref{theo:local_limit2}. As the tilting parameter $\vartheta$ moves from $0$ to $+\infty$, the bell-shaped curve moves from left to right.  Dashed black curve is the envelope of the maxima of all bell-shaped-curves.
}
\label{fig:stirlin_tilt_1_2}
\end{figure}

\begin{corollary}\label{cor:clt1}
For $i\in \{1,2\}$ and all $\vartheta >0$, and for $i=3$ and $\vartheta\geq 1$, the following central limit theorems hold:
$$
\frac{X^{(i)}_{n,\vartheta n} - \mu_i(\vartheta) n}{\sigma_i(\vartheta) \sqrt{n}} \todistr \mathcal N(0,1).
$$
\end{corollary}

A standard application of the G\"{a}rtner-Ellis theorem yields the large deviation principle for $(X^{(i)}_{n,\vartheta n}/n)_{n\in\N}$, $i=1,2,3$. The rate functions are just the Legendre-Fenchel transformations of the differentiable functions $\varphi_i$, $i=1,2,3$. Note that, for $i=1,2$, $\varphi^{\prime}_i(z;\vartheta)=\mu_i(\vartheta \eee^z)$. Thus, with $f^{\leftarrow}$ denoting the inverse  of a function $f$,
$$
(\varphi^{\prime}_i)^{\leftarrow}(t)=\log \mu_i^{\leftarrow}(t)-\log\vartheta,\quad i=1,2,\quad t\in (0,1).
$$
Simple calculations show that
\begin{equation}\label{eq:LDP_inverses}
\mu_1^{\leftarrow}(t)=-\frac{t}{t+W_{-1}(-t\eee^{-t})},\quad \mu_2^{\leftarrow}(t)=\frac{t}{1+tW_0(-t^{-1}\eee^{-1/t})}-t,
\quad
t\in (0,1),
\end{equation}
see Section~\ref{appendix_Lambert} for the definition of $W_{-1}$ which is also a branch of the Lambert $W$-function.
\begin{corollary}\label{cor:LDP}
Fix $\vartheta>0$. The sequences of random variables $(X^{(1)}_{n,\vartheta n}/n)_{n\in\N}$ and $(X^{(2)}_{n,\vartheta n}/n)_{n\in\N}$ satisfy large deviations principles with speed $n$ and with the following rate functions:
$$
I_i(t;\vartheta) =
\begin{cases}
t \log \mu_i^{\leftarrow}(t) - L_i(\mu_i^{\leftarrow}(t)) - (t \log \vartheta - L_i(\vartheta)),
&\text{ if } t\in (0,1),\\
L_1(\theta),
&\text{ if } t=0, i=1,\\
+\infty,
&\text{ if } t=0, i=2,\\
L_1(\vartheta) - 1 - \log \vartheta,
&\text{ if } t=1, i\in \{1,2\},\\
%????,
%&\text{ if } t=1, i=2,\\
+\infty, &\text{ if } t\in \R\setminus [0,1],
\end{cases}
$$
for $i\in \{1,2\}$, where $\mu_i^{\leftarrow}$ are given explicitly by~\eqref{eq:LDP_inverses}.
\end{corollary}

Large deviations for $X_{n,\vartheta n}^{(1)}$ have been obtained in~\cite[Theorem~4.4]{feng_ewens}. A slightly more involved calculation leads to large deviations estimates for $X_{n,\vartheta n}^{(3)}$.

%Note that if $\vartheta\geq 1$, then $X_{n,\vartheta n}^{(3)}$ is concentrated on $\{1,2,\ldots,n\}$, whereas if $\vartheta<1$, then $X_{n,\vartheta n}^{(3)}$ is concentrated on $\{1,2,\ldots,n\vartheta\}$. Thus, the shape of the rate function is different for $\vartheta\geq 1$ and $\vartheta<1$.

\begin{corollary}\label{cor:LDP_X_3}
Fix $\vartheta\geq 1$. The sequence of random variables $(X^{(3)}_{n,\vartheta n}/n)_{n\in\N}$ satisfy large deviations principles with speed $n$ and with the following rate function:
$$
I_3(t;\vartheta) =
\begin{cases}
(t-1)\log \mu_2^{\leftarrow}(t) - \frac{1}{t+\mu_2^{\leftarrow}(t)}+(\vartheta-t)\log (\vartheta-t)+1-(\vartheta-1)\log\vartheta,
&\text{ if } t\in (0,1),\\
(\vartheta-1)\log(\vartheta-1)+1+(\vartheta-1)\log\vartheta, &\text{ if } t=1,\\
+\infty, &\text{ if } t\notin (0,1].
\end{cases}
$$
\end{corollary}
\begin{remark}\label{rem:x_3_theta<1}
The claims of Lemma~\ref{lem:linear_growth}, Corollary~\ref{cor:clt1} and Corollary~\ref{cor:LDP_X_3} remain true also for $i=3$ and $\vartheta\in (0,1)$ upon appropriate interpretation in case when $n\vartheta$ is not an integer. This is a consequence of the fact that Theorem~\ref{thm:mod_phi_stirling3} holds true for all $\vartheta>0$. More precisely, if $\vartheta\in (0,1)$ and $n\vartheta$ is not an integer, then $\E X_{n,\vartheta n}^{(3)}$ and $\Var X_{n,\vartheta n}^{(3)}$ in Lemma~\ref{lem:linear_growth} should be interpreted as the first two derivatives of $z\mapsto \log \mathcal{G}_{n,\vartheta n}^{(3)}(\eee^z)$ at $z=0$.
%$\left(\mathcal{G}_{n,\vartheta n}^{(3)}\right)^{\prime}(1)$
%and
%$\left(\mathcal{G}_{n,\vartheta n}^{(3)}\right)^{\prime\prime}(1)+\left(\mathcal{G}_{n,\vartheta %n}^{(3)}\right)^{\prime}(1)-\left(\left(\mathcal{G}_{n,\vartheta n}^{(3)}\right)^{\prime}(1)\right)^2$,
%respectively.
The central limit theorem in Corollary~\ref{cor:clt1} is understood as the convergence of properly rescaled Laplace transforms to the Laplace transform of the standard normal law in a small neighborhood of the origin. The rate function in Corollary~\ref{cor:LDP_X_3}, for $\vartheta\in (0,1)$, takes the form
$$
I_3(t;\vartheta) =
\begin{cases}
(t-1)\log \mu_2^{\leftarrow}(t) - \frac{1}{t+\mu_2^{\leftarrow}(t)}+(\vartheta-t)\log (\vartheta-t)+1-(\vartheta-1)\log\vartheta,
&\text{ if } t\in (0,\vartheta),\\
(\vartheta-1)\log \mu_2^{\leftarrow}(\vartheta) - \frac{1}{\vartheta+\mu_2^{\leftarrow}(\vartheta)}+1-(\vartheta-1)\log\vartheta, &\text{ if } t=\vartheta.\\
%+\infty, &\text{ if } t\notin (0,\vartheta].\\
\end{cases}
$$
\end{remark}

\begin{figure}[t]
\begin{center}
\includegraphics[width=0.25\textwidth ]{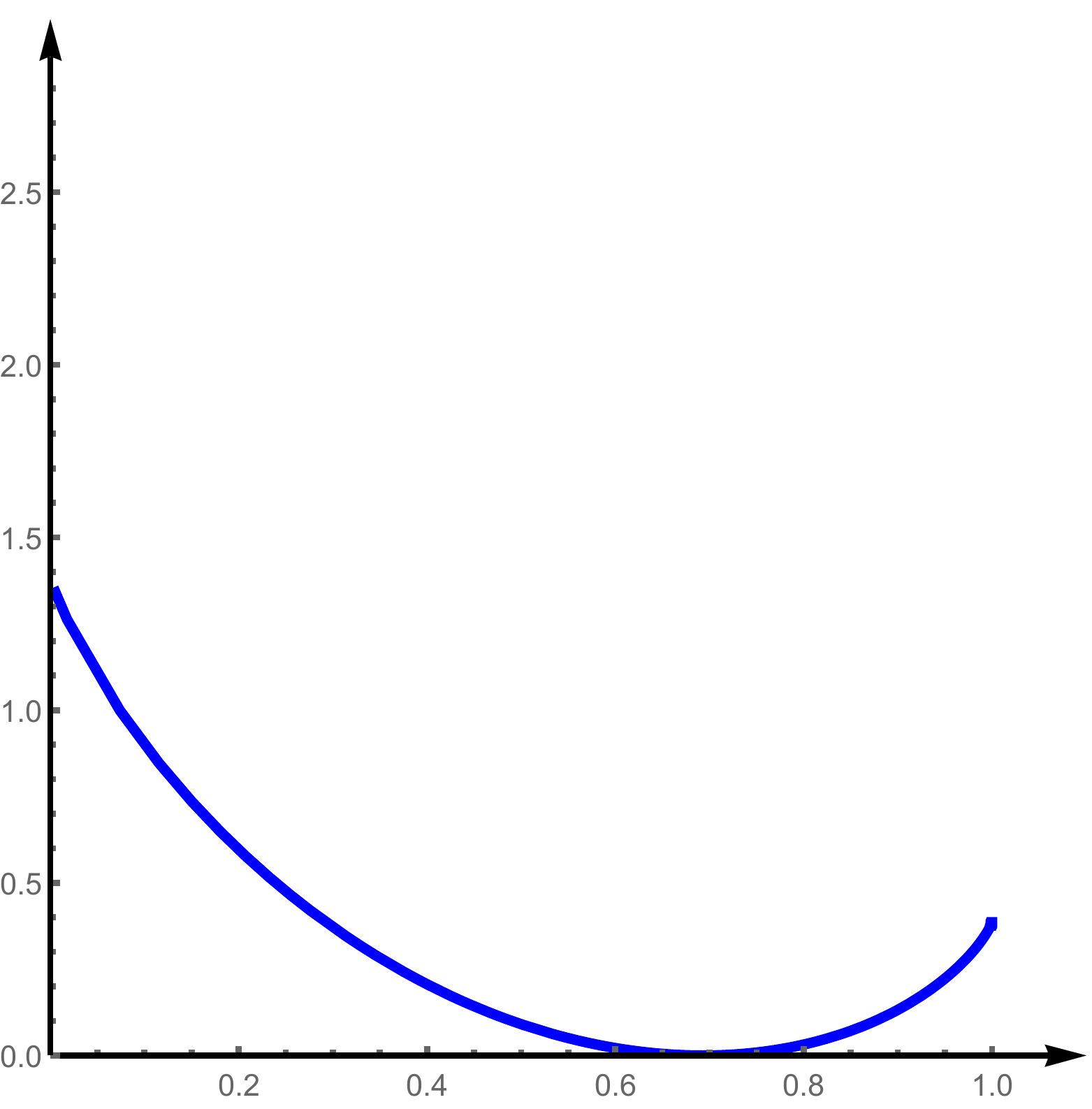}
\includegraphics[width=0.25\textwidth ]{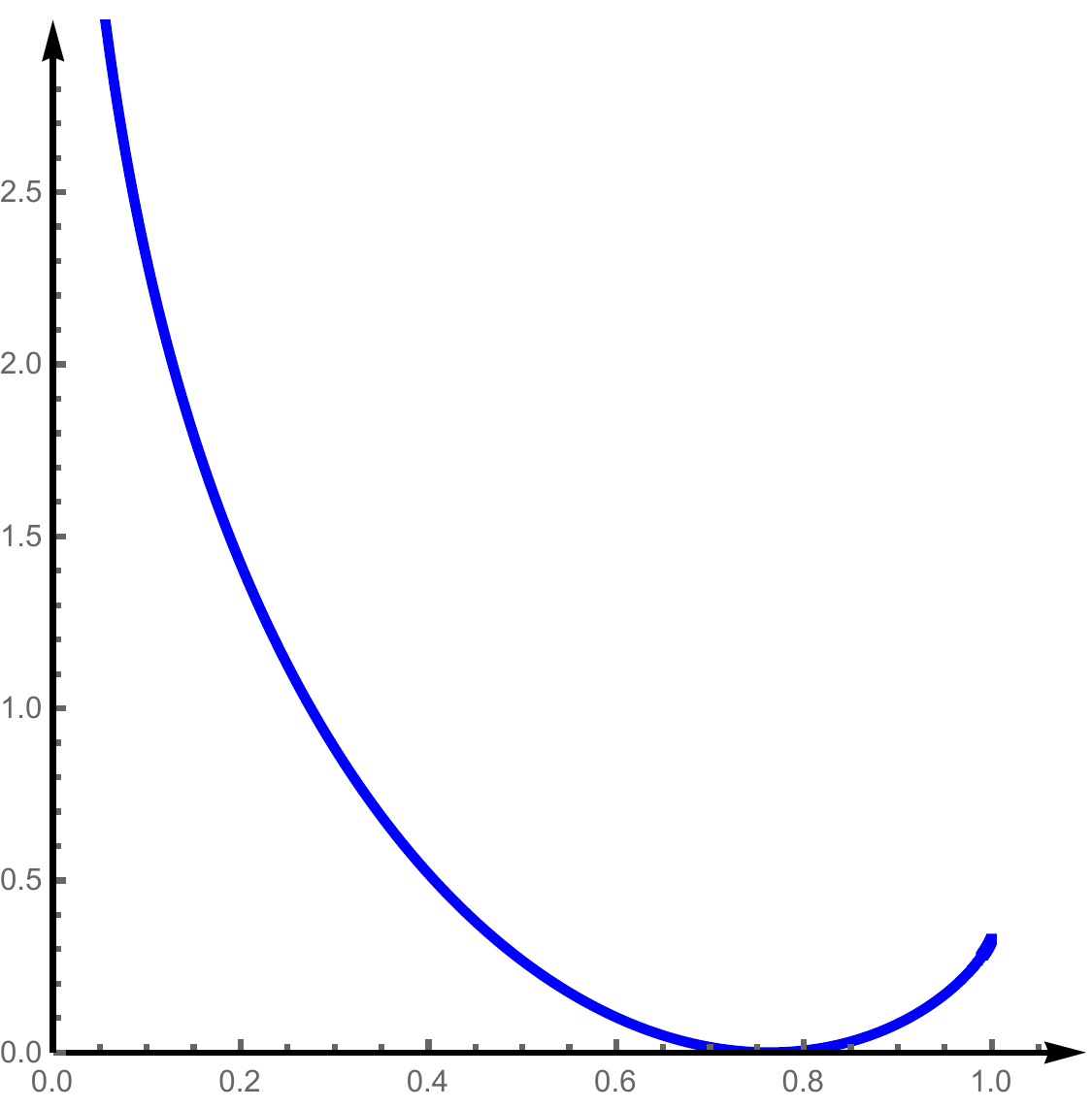}\\
\includegraphics[width=0.25\textwidth ]{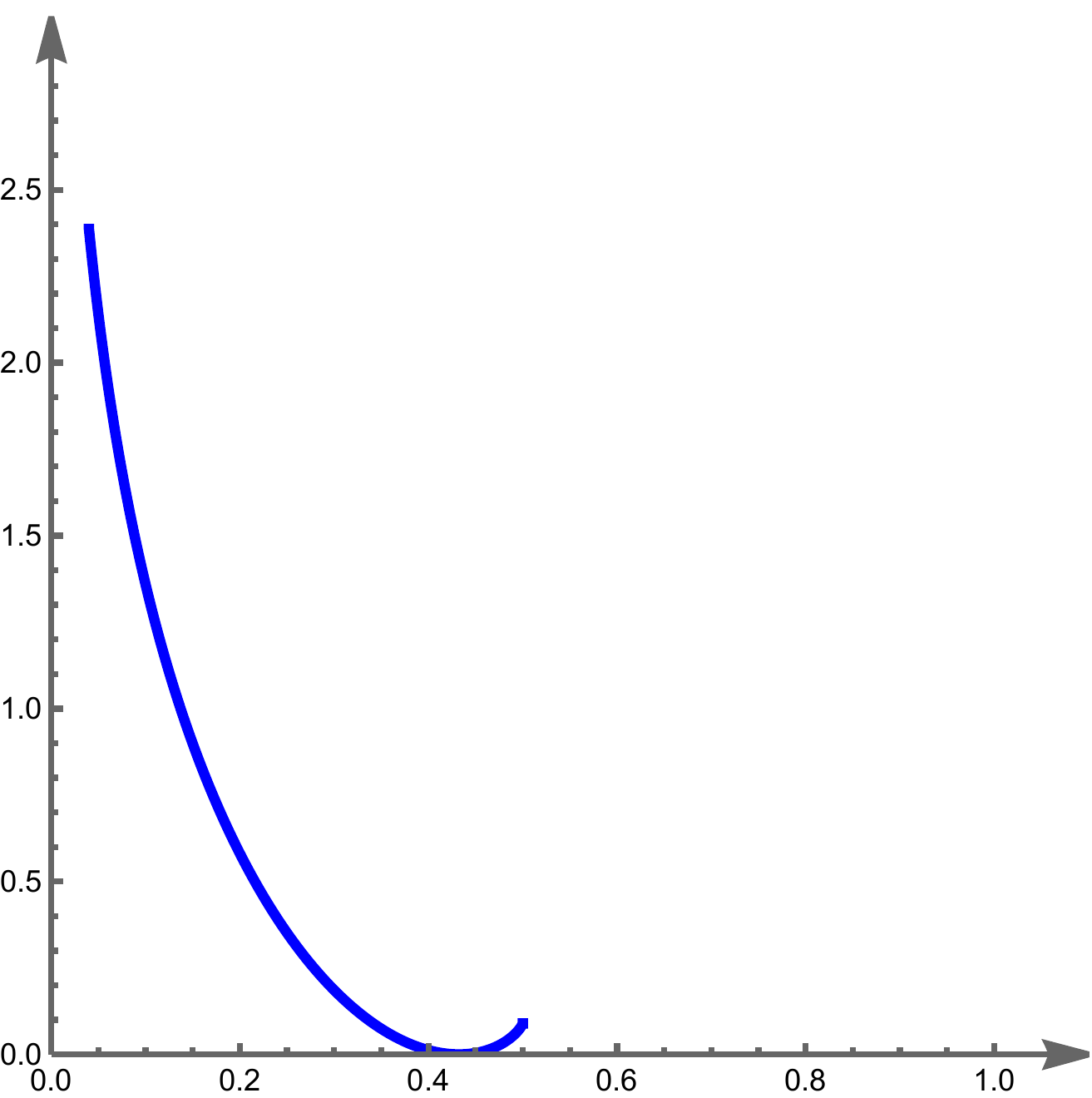}
\includegraphics[width=0.25\textwidth ]{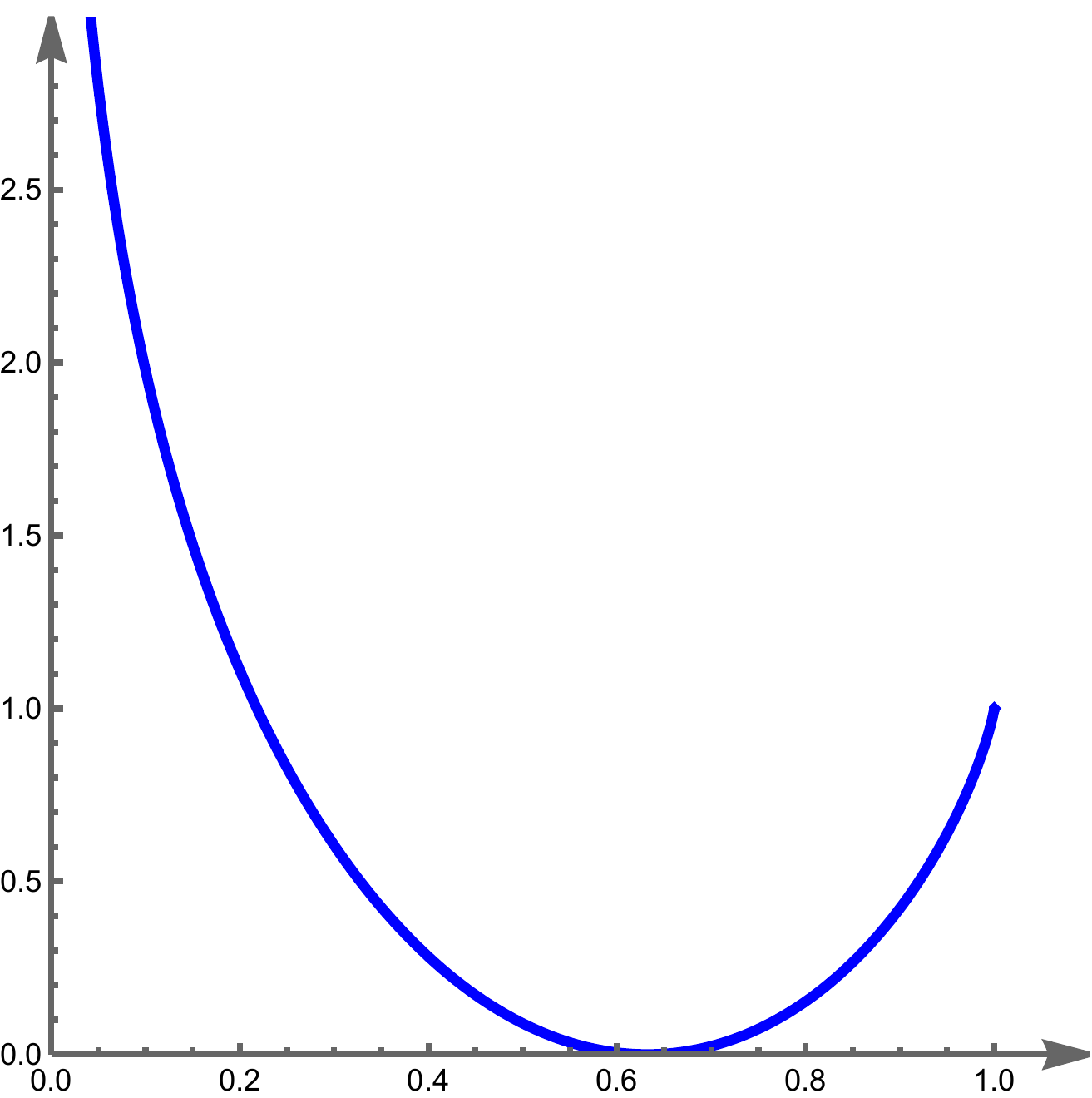}
\includegraphics[width=0.25\textwidth ]{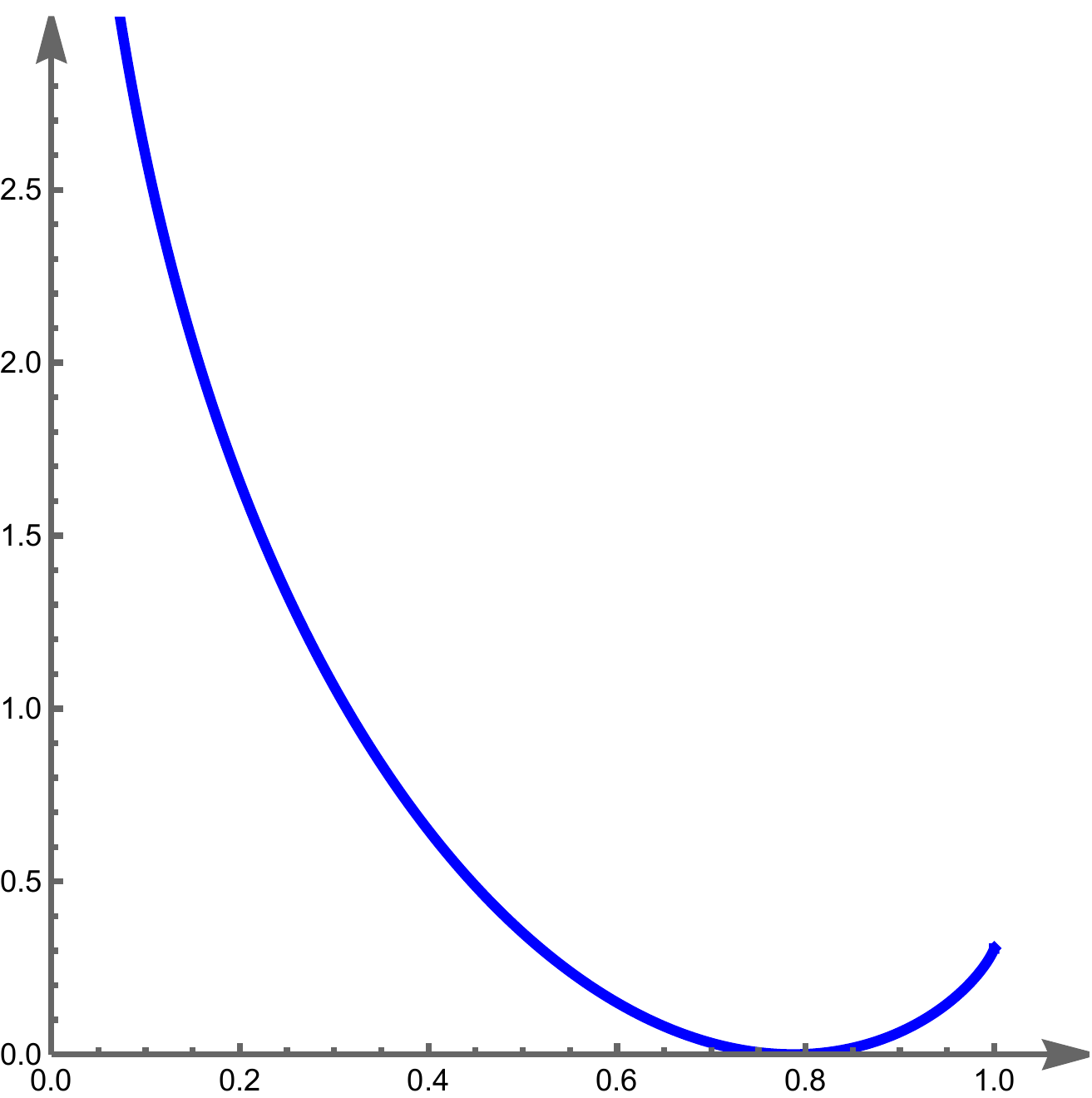}
\end{center}
\caption
{
Rate functions $I_1(t;1)$ (top left), $I_2(t;1)$ (top right), $I_3(t;1/2)$ (bottom left), $I_3(t;1)$ (bottom middle) and  $I_3(t;2)$ (bottom right); see Corollary~\ref{cor:LDP} and Corollary~\ref{cor:LDP_X_3}.
}
\label{fig:rate_function}
\end{figure}

The proofs of Theorems~\ref{thm:mod_phi_stirling1}, ~\ref{thm:mod_phi_stirling2} and~\ref{thm:mod_phi_stirling3} are given in Section~\ref{sec:mod-phi-proofs}.

Let us finally mention that the mod-$\varphi$ convergence, that we prove for $X_{n,\vartheta n}^{(i)}$, $i=1,2,3$, can be used to derive various asymptotic estimates for the Stirling numbers of both kinds. However, we refrain from presenting here the corresponding formulas since they follow, for the most part, from the known results available in the vast body of literature on the asymptotics of Stirling numbers, see, for example, \cite{bender,bleick_wang,canfield,hwang_diss,hwang_stirling,louchard_first_kind,louchard,moser_wyman,moser_wyman_stirling2,sachkov_book,temme,wilf}.

\subsection{Mod-\texorpdfstring{$\varphi$}{phi} convergence and limit theorems for zeros of generating functions}\label{sec:finite_free}

In this section we shall explain how the mod-$\varphi$ convergence is related to the convergence of empirical measures of zeros of the corresponding generating functions and derive a counterpart of Elbert's result~\eqref{eq:elbert_measures}-\eqref{eq:elbert_limit} for the zeros of $\mathcal{G}_{n,\vartheta n}^{(3)}$. The skipped proofs are postponed to Section~\ref{sec:finite_free_proof}.

Suppose that $X_n$ is a random variable which takes values in $\{0,1,2,\ldots\}$ and has a bounded support, for every fixed $n\in\N$. Then, the generating function of $X_n$, defined by
$$
\mathcal{G}_n(t):=\E t^{X_n},\quad t\in\C,\quad n\in\N,
$$
is a polynomial. Recall that $\Zeros(\mathcal{G}_n)$ denotes the multiset of zeros of $\mathcal{G}_n$ counted with multiplicities.

A sequence $(\nu_n)_{n\in \N}$ of finite measures on $\R$ converges vaguely to a finite measure $\nu$ if $\lim_{n\to\infty} \int_\R f\dd \nu_n = \int_\R f \dd \nu$ for every continuous function $f:\R\to \R$ with a compact support. If, additionally,  $\sup_{n\in\N}\nu_n(\R)<\infty$, then the above condition holds for all continuous $f$ with $\lim_{t\to \pm \infty} f(t) = 0$.

\begin{proposition}\label{prop:zeros}
Suppose that $\mathcal{G}_n$ has only real (hence, nonpositive) zeros. Suppose further that, for some domain $\mathcal{D}\subset \C$, analytic functions $\varphi$ and $\Psi$ and a sequence $(w_n)_{n\in \N}$, the sequence $(X_n)_{n\in\N}$ converges mod-$\varphi$ with the corresponding parameters, that is,~\eqref{eq:mod_phi_def_conv} holds locally uniformly on $\mathcal{D}$. Finally, suppose that $\sup_{n\in\N}{\rm deg} \, (\mathcal{G}_n)/w_n<\infty$. Then, the sequence
$$
\mathcal{Z}_n(\cdot):=\frac{1}{w_n}\sum_{x\in \Zeros(\mathcal{G}_n)}\delta_{-x}(\cdot),\quad n\in\N,
$$
of finite measures on $[0,+\infty)$, converges vaguely to a finite  limit measure $\mathcal{Z}$ supported by $[0,+\infty)$ with the Stieltjes transform
$$
\int_{[0,+\infty)}\frac{\mathcal{Z}(\dd x)}{t-x}=\frac{\varphi'(\log (-t))}{t},\quad t\in\C\setminus [0,+\infty).
$$
In particular, the function $t\mapsto \varphi'(\log (-t))/t$ admits analytic continuation to $\C\setminus [0,+\infty)$.
\end{proposition}

Using Proposition~\ref{prop:zeros} and Theorem~\ref{thm:mod_phi_stirling2} applied with $\vartheta=1$ we recover Elbert's result \eqref{eq:elbert_measures}-\eqref{eq:elbert_limit} for the zeros of the Touchard polynomials. Indeed,
$$
\frac{1}{z}\varphi_2^{\prime}(\log (-z);1)=\frac{1}{z}\left(\frac{1}{W_0(-1/z)}+z\right)=\frac{1}{zW_0(-1/z)}+1=1-\eee^{W_0(-1/z)},\quad z\in\C\setminus [0,\eee],
$$
and since in this case all $\mathcal{Z}_n$ and the limit measure $\mathcal{Z}$ have the total mass $1$, vague convergence secured by Proposition~\ref{prop:zeros} is equivalent to the weak convergence.

We now aim at applying Proposition~\ref{prop:zeros} to the sequence $(X_{n,\theta}^{(3)})_{n\in\N}$. It turns out that the zeros of the generating function of $X_{n,\theta}^{(3)}$ are all real for all admissible choices of the parameter $\theta$, that is, for integer $\theta>0$ or real $\theta$ satisfying $\theta>n-1$. In order to prove  this statement we shall recall some notions from finite free probability.

\subsubsection{Finite free multiplicative convolutions and the generating function of $X_{n,\theta}^{(3)}$.}

The \textit{finite free multiplicative convolution} $\boxtimes_{n}$ is a bilinear operation on the space $\C_{n}[x]$ of polynomials of degree at most $n$ which is defined  as follows:
\begin{equation}\label{eq:finite_free_conv_def}
\left(\sum_{k=0}^{n} \alpha_k x^k \right)\boxtimes_{n}\left(\sum_{k=0}^{n} \beta_k x^k \right)
=
\sum_{k=0}^{n} (-1)^{n-k}  \frac{\alpha_k \beta_k}{\binom {n}{k}} x^k;
\end{equation}
see~\cite{marcus,marcus_spielman_srivastava}.
The polynomial $(x-1)^{n}$ is the unit element under $\boxtimes_{n}$, that is $(x-1)^{n}\boxtimes_{n} p(x) =  p(x)$ for all $p \in \C_{n}[x]$.

\begin{proposition}\label{prop:rep_as_free_conv}
For all $\theta\in \R$ and an arbitrary polynomial $\sum_{k=0}^n a_k x^k \in \C_n[x]$ the following holds:
$$
\sum_{k=0}^n a_k \theta (\theta - 1) \cdots (\theta - k +1)x^k
=
\left(\sum_{k=0}^n a_kx^k\right) \boxtimes_{n} \left( n!x^n L_n^{(\theta - n)} (1/x)\right),
$$
where the generalized Laguerre polynomials $L_n^{(\alpha)}$ are defined, for $n\in \N_0$ and $\alpha\in \R$, by
$$
L_n^{(\alpha)} (x)
=
\sum_{k=0}^n (-1)^k \binom {n+\alpha}{n-k} \frac{x^k}{k!}.
$$
\end{proposition}
\begin{proof}
The observation becomes evident if we write
$$
n!x^n L_n^{(\theta - n)} (1/x)
%=
%\sum_{k=0}^n (-1)^{n-k}  \frac{\theta! n!}{k! (n-k)! (\theta-k)!} x^k
=
\sum_{k=0}^n (-1)^{n-k}\theta (\theta - 1) \cdots (\theta - k +1) \binom nk  x^k
$$
and use~\eqref{eq:finite_free_conv_def}.
\end{proof}
In particular, Proposition~\ref{prop:rep_as_free_conv} yields the representation
\begin{equation}\label{eq:g_3_as_free_conv}
\mathcal{G}_{n,\theta}^{(3)}(-x)= \theta^{-n}(T_n(-x) \boxtimes_{n} (n!x^n L_n^{(\theta - n)} (1/x))),\quad \theta\in\R,\quad n\in\N.
\end{equation}

\begin{proposition}\label{prop:real_roots_third_type}
If the polynomial $\sum_{k=0}^n a_k x^k$ has only nonnegative zeros, then $\sum_{k=0}^n a_k \theta (\theta-1) \ldots (\theta - k + 1) x^k$ also has only nonnegative zeros, for all real $\theta > n-1$ and all integer $\theta \in \{0,1,\ldots\}$.
\end{proposition}
\begin{proof}
It is known that $p\boxtimes_n q$ has only nonnegative zeros  provided that $p\in \C_n[x]$ and $q\in \C_n[x]$ both have only nonnegative zeros, see~\cite[Theorem~1.6]{marcus_spielman_srivastava}.  Also, it is known that all zeros of $L_n^{(\alpha)}$ are real and nonnegative for all real $\alpha > -1$ and all integer $\alpha \in \{-n,\ldots, -1\}$, see Example~2.8 in \cite{arizmendi_etal}. If $\alpha>-1$, this is true because $L_n^{(\alpha)}$ are orthogonal with respect to the finite measure $x^{\alpha} \eee^{-x}$, $x>0$, whereas for integer $\alpha\in \{-n,\ldots, -1\}$ this follows from the identity
\begin{equation}\label{eq:laguerre_duality}
\frac{(-x)^i}{i!}L_n^{(i-n)}(x)=\frac{(-x)^n}{n!}L_i^{(n-i)}(x),\quad x\in\C,\quad 0\leq i<n.
\end{equation}
Hence, the polynomial $x^n L_n^{(\theta - n)} (1/x)$ has only nonnegative  zeros and the claim follows from Proposition~\ref{prop:rep_as_free_conv}.
\end{proof}

\begin{remark}
For integer $\theta\in\{0,1,\ldots\}$ Proposition~\ref{prop:real_roots_third_type} is a classical result originally due to Schur, see Problems~155 and~156 in Part V of Volume II~\cite{polya_book}. For  related results see Section 3 in \cite{pitman_probab_bounds} and Section 3.5 in~\cite{brenti1988unimodal}.
\end{remark}

Taking into account Harper's result on the nonpositivity of zeros of the Touchard polynomials, Proposition~\ref{prop:real_roots_third_type} and formula~\eqref{eq:g_3_as_free_conv} imply that for all real $\theta > n-1$ and all integer $\theta \in \N$, all zeros of the polynomials
$\mathcal{G}_{n,\theta}^{(3)}$ are nonpositive.
From this fact we immediately obtain the following observation, which is originally due to Vatutin and Mikhailov~\cite{vatutin_mikhailov}.
Another related result can be found in~\cite{vatutin_ascending_segments}.
\begin{proposition}[Lemma 1 in~\cite{vatutin_mikhailov}]
The random variable $X_{n,\theta}^{(3)}$ which counts the number of occupied boxes when $n$ balls are allocated equiprobably and  independently among $\theta\in \N$ boxes, can be represented as a sum of independent (but not identically distributed) Bernoulli variables.
\end{proposition}

\subsubsection{Limit theorem for the empirical distribution of zeros of the generating function of $X_{n,\vartheta n}^{(3)}$}
%In view of Proposition~\ref{prop:real_roots_third_type} and the fact that all zeros of $T_n$ are real we conclude that all zeros of $\mathcal{G}_{n,\theta}^{(3)}$ are real (hence, nonpositive) for integer $\vartheta$ or real $\vartheta>n-1$. Thus,
Using Proposition~\ref{prop:zeros} in conjunction with Theorem~\ref{thm:mod_phi_stirling3} we obtain the next result.
\begin{proposition}\label{prop:convergence_of_zeros_of_x_3}
Assume that $\vartheta\geq 1$. The sequence of probability measures on $[0,+\infty)$ defined  by
$$
\frac{1}{n}\sum_{x\in \Zeros(\mathcal{G}_{n,n\vartheta}^{(3)})}\delta_{-x}(\cdot),\quad n\in\N,
$$
converges weakly to a probability measure $\mathcal{Z}^{\vartheta}$ with the Stieltjes transform
\begin{equation}\label{eq:formula_st_x_3_measure}
\int_{[0,\infty)}\frac{\mathcal{Z}^{\vartheta}(\dd x)}{t-x}~=~\frac{\varphi_3^{\prime}(\log (-t);\vartheta)}{t},\quad t\in \C\setminus [0,\infty),
\end{equation}
where
$$
\frac{\varphi_3^{\prime}(\log (-t);\vartheta)}{t}
=
%\frac{1}{t}\left(\vartheta-\frac{\vartheta^2 W_0(\vartheta^{-1}\eee^{-1/\vartheta}(-1/t-1))}{(t+1)(1+\vartheta W_0(\vartheta^{-1}\eee^{-1/\vartheta}(-1/t-1))}\right),\quad t\in \C\setminus [0,\infty).
\frac{\vartheta}{t} -\frac{\vartheta^2}{t(t+1)} \frac{W_0\left(\frac {- t^{-1} - 1}{\vartheta \eee^{1/\vartheta}}\right)}{1 + \vartheta W_0\left(\frac {- t^{-1} - 1}{\vartheta \eee^{1/\vartheta}}\right)},
\quad t\in \C\setminus [0,\infty).
$$
\end{proposition}
Note that Proposition~\ref{prop:zeros} secures only the vague convergence, however, since all the measures involved are probability measures, the weak convergence holds as well. The fact that $\mathcal Z^{\vartheta}$ is a probability measure for $\vartheta\geq 1$ will be justified in the proof of Proposition~\ref{prop:density_properties}.
\begin{remark}
Note that formula~\eqref{eq:formula_st_x_3_measure} defines
$\mathcal{Z}^{\vartheta}$ also for
$\vartheta\in (0,1)$. In this case it is an improper probability measure with the total mass $\vartheta$, see Proposition~\ref{prop:density_properties} below. Proposition~\ref{prop:convergence_of_zeros_of_x_3} remains valid in the following sense. Let $\theta_n$ be a sequence of integers such that $\theta_n\sim \vartheta n$, as $n\to\infty$, for some $\vartheta\in (0,1)$. Then the sequence of finite measures on $[0,+\infty)$ defined  by
$$
\frac{1}{n}\sum_{x\in \Zeros(\mathcal{G}_{n,\theta_n}^{(3)})}\delta_{-x}(\cdot),\quad n\in\N,
$$
converges weakly to $\mathcal{Z}^{\vartheta}$. This will be justified at the end of Section~\ref{sec:saddle_point_mod_phi}, see Remark~\ref{rem:uniformity_in_theta}.
\end{remark}

The measure $\mathcal{Z}^{\vartheta}$ turns out to be absolutely continuous with respect to the Lebesgue measure, see Figure~\ref{fig:densities_allocation}.

\begin{proposition}\label{prop:density_properties}
The  limit measure $\mathcal{Z}^{\vartheta}$ is a probability measure if $\vartheta\geq 1$, and is a finite measure with the total mass $\vartheta$ if $\vartheta\in (0,1)$. The density of $\mathcal{Z}^{\vartheta}$ is given by
\begin{equation}\label{eq:density_alloc_explicit}
g_{\vartheta}(t) = \frac {\vartheta^2}{\pi t (t+1)} \cdot  \Im  \frac{W_0\left(\frac {- t^{-1} - 1 + \iii 0}{\vartheta \eee^{1/\vartheta}}\right)}{1 + \vartheta W_0\left(\frac {- t^{-1} - 1 + \iii 0}{\vartheta \eee^{1/\vartheta}}\right)},
\qquad t>0.
\end{equation}
For all $\vartheta>0$,
\begin{equation}\label{eq:density_alloc_asympt_0}
g_\vartheta(t) \sim \frac{1}{t(\log t)^2},
\qquad
t\downarrow 0.
\end{equation}
For $\vartheta =1$,
\begin{equation}\label{eq:density_alloc_asympt_infty_vartheta1}
g_{\vartheta}(t) \sim  \frac{1}{\sqrt 2 \, \pi \, t^{3/2}},
\qquad
t\to +\infty.
\end{equation}
For $\vartheta>0$, $\vartheta\neq 1$, the density vanishes outside the interval $(0, m_\vartheta)$ with $m_\vartheta := 1/(\vartheta \eee^{(1/\vartheta) - 1} - 1)$ and
\begin{equation}\label{eq:density_alloc_asympt_infty_vartheta_neq1}
g_{\vartheta}(m_\vartheta - \eps) \sim \frac{\vartheta^{3/2} \sqrt{2\eee^{1-(1/\vartheta)}}}{\pi m_\vartheta^2 (m_\vartheta + 1) (\vartheta-1)^2} \cdot \sqrt \eps,
\qquad \eps\downarrow 0.
\end{equation}
\end{proposition}

\begin{figure}[t]
\begin{center}
\includegraphics[width=0.45\textwidth]{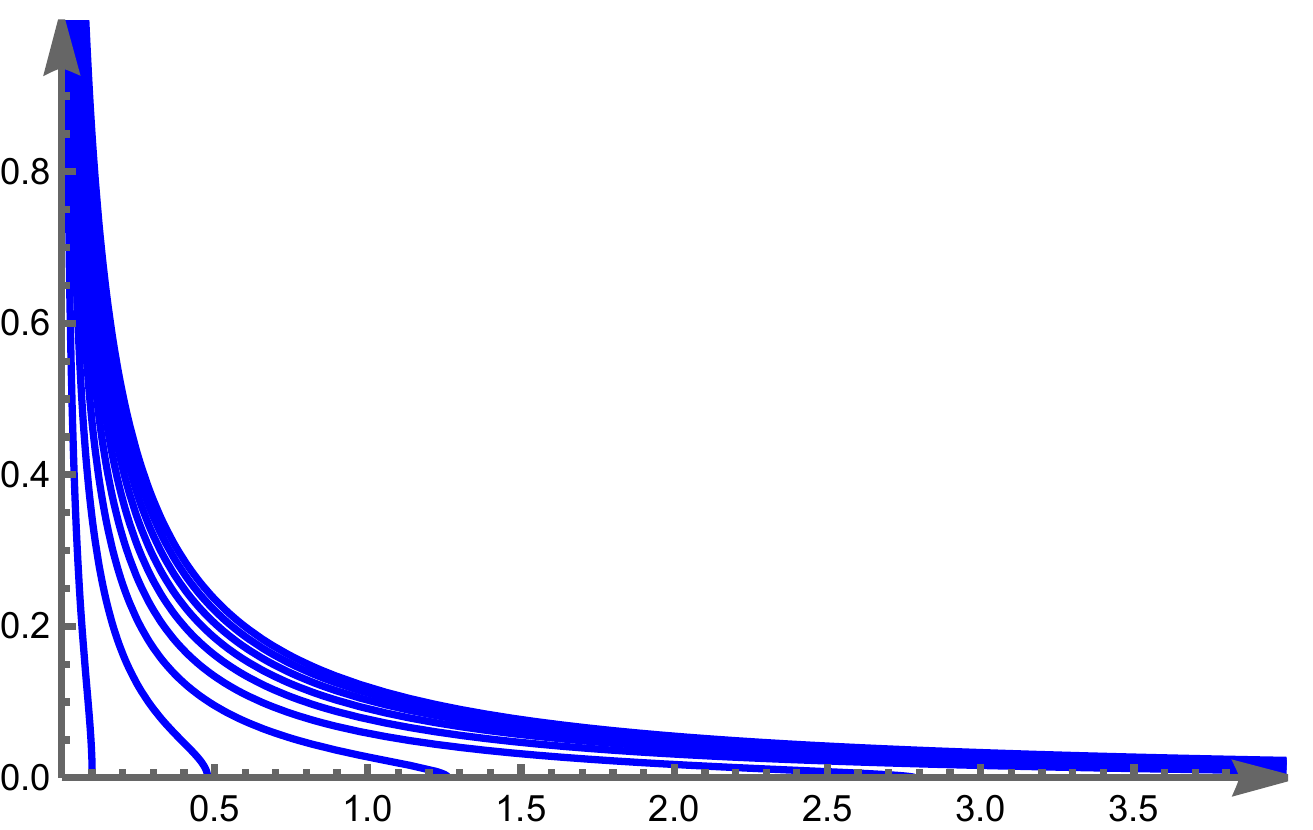}
\includegraphics[width=0.45\textwidth]{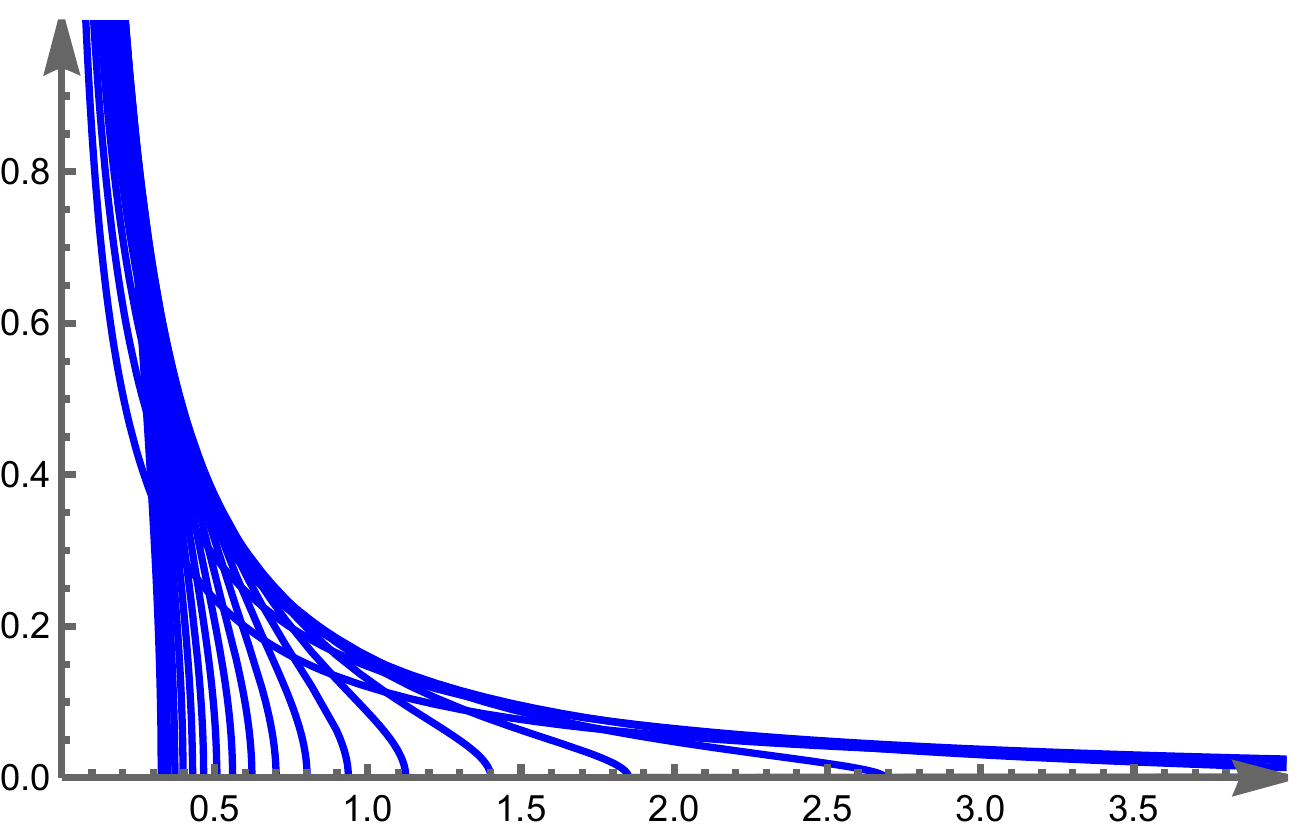}
\end{center}
\caption
{
The densities of $\mathcal{Z}^{\vartheta}$. Left: $\vartheta= 0.2,0.3,\ldots, 1$. Right: $\vartheta= 1, 1.5, 2,\ldots, 10$.
}
\label{fig:densities_allocation}
\end{figure}

%\begin{remark}
%It seems that the support of the density exactly  equals $(0,m_\vartheta)$, for all $\vartheta>0$, where we put $m_1:=+\infty$. It seems also that the density is strictly monotone decreasing (and non-vanishing) on $(0,m_\vartheta)$, for all $\vartheta>0$; see Figure~\ref{fig:densities_allocation}.
%\end{remark}
\begin{remark}
The asymptotics in~\eqref{eq:density_alloc_asympt_0} and the square-root character of the singularity in~\eqref{eq:density_alloc_asympt_infty_vartheta_neq1} are the same as for the density of Elbert's distribution~\eqref{eq:elbert_limit} supported by the interval $(0,\eee)$. If we denote the density of $\rho$ by $f_{\text{Elb}}$, then it is known~\cite[Theorem~2.4]{elbert2} that
$$
f_{\text{Elb}}(\eps) \sim \frac 1 {\eps(\log \eps)^2},
\quad
f_{\text{Elb}}(\eee - \eps) \sim \frac{\sqrt {2\eps}}{\eee^{3/2} \pi},
\quad\eps\downarrow 0.
$$
\end{remark}
\begin{remark}
Formula~\eqref{eq:density_alloc_asympt_infty_vartheta1} suggests that for $\vartheta=1$, the smallest root of the generating function $\mathcal{G}_{n,n}^{(3)}$ is asymptotically equivalent to $-c\cdot n^2$, as $n\to\infty$, for some $c>0$. This conjecture, supported by numerical simulations,  is obtained by noting that  $\int_{c n^2}^\infty \dd t /(t^{3/2})$ is of order $1/n$, as $n\to\infty$.
%See Poly_IID_coeff.nb for the simulation of roots
%For $\vartheta\neq 1$, the smallest root of the generating polynomial $\mathcal{G}_{n,\vartheta n}^{(3)}$ converges to $-m_{\vartheta}$, without any normalization.
\end{remark}

There is an alternative representation of $\mathcal{Z}^{\vartheta}$.  Recall from~\cite[Example~3.3.5]{hiai_petz_book} or~\cite[p.~35]{voiculescu_dykema_nica_book} that the Marchenko--Pastur distribution $\mu_\vartheta^{MP}$ (also called the free Poisson distribution) with parameter $\vartheta \geq 1$ has the density
$$
x\mapsto \frac 1 {2\pi x} \sqrt{((\sqrt \vartheta +1)^2 - x)(x -(\sqrt \vartheta - 1)^2 )}\1_{\{x\in [(\sqrt \vartheta - 1)^2, (\sqrt \vartheta + 1)^2]\}}.
$$
For $0< \vartheta < 1$, the Marchenko-Pastur distribution $\mu_\vartheta^{MP}$ has the same density on the same interval plus an atom at $0$ with weight $1-\vartheta$.

\begin{proposition}\label{prop:convolution_of_MP_andElbert}
If $\vartheta>1$, then $\mathcal{Z}^{\vartheta}$ is equal to the free multiplicative convolution, see~\cite[p.~30]{voiculescu_dykema_nica_book}, of the probability measure $\rho$
defined by~\eqref{eq:elbert_limit}, and the distribution of $1/\xi_\vartheta$, where $\xi_\vartheta$ has the Marchenko-Pastur distribution $\mu_\vartheta^{MP}$.
\end{proposition}

\begin{proof}%[Proof of Proposition~\ref{prop:convolution_of_MP_andElbert}]
It follows from the definition of $\boxtimes_n$ that $p(cx)\boxtimes_n q(x/c) = p(x) \boxtimes_n q(x)$ for every constant $c\neq 0$ and arbitrary $p,q\in \C_n[x]$. Thus, by formula~\eqref{eq:g_3_as_free_conv},
$$
\theta^n \mathcal{G}_{n,\theta}^{(3)}(-x)
=
\left(\sum_{k=0}^n \stirlingsec{n}{k} (-xn)^k \right) \boxtimes_{n} \left( n! n^{-n} x^n L_n^{(\theta - n)} (n/x)\right)
=:
p_n(x) \boxtimes_n q_n(x).
$$
It is known that the empirical distribution of zeros of $p_n$ converges to the measure $\rho$ defined by~\eqref{eq:elbert_limit}.   Also, it is known from~\cite[Theorem~3.1 (a)]{dette_studden} that the empirical distribution of zeros of $y\mapsto L_n^{(\vartheta n - n)} (ny)$ converges weakly to the Marchenko--Pastur distribution $\mu_\vartheta^{MP}$. Moreover, by~\cite[Theorem~4.4(a)]{dette_studden}, the smallest and the largest zero of this polynomial converge to $(\sqrt \vartheta - 1)^2>0$, respectively,  $(\sqrt \vartheta + 1)^2<\infty$.  Hence, the empirical measure of zeros of $ q_n(x) = n! n^{-n} x^n L_n^{(\theta - n)} (n/x)$ converges weakly to the distribution of $1/\xi_\vartheta$. The proof concludes by an appeal to Theorem~1.4 from~\cite{arizmendi_etal}.
\end{proof}

\begin{remark}
Proposition~\ref{prop:convolution_of_MP_andElbert} remains valid for $\vartheta\in (0,1)$, if regard the distribution of $1/\xi_{\vartheta}$ as an improper probability measure on $\R$ with the total weight $\vartheta$ (and ignore the atom at $+\infty$ of the weight $1-\vartheta$ which occurs due to the atom at $0$ of the Marchenko--Pastur distribution). The result on the distribution of zeroes of the polynomials $L_n^{(\vartheta n - n)} (ny)$ follows from~\cite[Theorem~3.1 (a)]{dette_studden}, by taking into account~\eqref{eq:laguerre_duality}.
\end{remark}

\section{Proofs for Section~\ref{subsec:def_mod_phi}}\label{sec:mod-phi-proofs}

\subsection{Proof of Theorems~\ref{thm:mod_phi_stirling1}, \ref{thm:mod_phi_stirling2} and~\ref{thm:mod_phi_stirling3}}\label{sec:saddle_point_mod_phi}

We start with the known asymptotic results for the functions $S_n$, $T_n$. Proposition ~\ref{prop:asympt_factorial1} which deals with the asymptotics of $S_n$ is easy and follows from the standard Stirling asymptotics. Proposition~\ref{prop:asympt_factorial2} is also known~\cite[Theorems 3.2 and 3.3]{elbert1}; see also~\cite{paris_touchard,paris_uniform,zhao} for related results.
Nevertheless, we find it instructive to formulate and prove them both as a preparation to the proof of Theorem~\ref{thm:mod_phi_stirling3}.

\begin{proposition}\label{prop:asympt_factorial1}
Let $\mathcal{D}^{\prime}_1 = \C \setminus (-\infty,0]$. Then with the function $L_1$ defined by~\eqref{eq:def_L_1} it holds
\begin{equation}\label{eq:asympt_factorial}
\frac{S_n(z n)}{n!}~=~
\frac 1 {\sqrt{2\pi n}} \sqrt{\frac{z}{z+1}} \eee^{n L_1(z)}\left(1+O\left(\frac{1}{n}\right)\right),\quad n\to\infty,
\end{equation}
locally uniformly in $z\in\mathcal{D}^{\prime}_1$.
\end{proposition}
\begin{proof}
Using the definition of $S_n$ we write
$$
\frac{S_n(z n)}{n!} = \frac{z n (z n+1)\cdots (z n +n-1)}{n!}
=
\frac{\Gamma((z+1)n)}{\Gamma(z n) n\Gamma(n)}.
$$
Recall the Stirling approximation for the gamma-function,
\begin{equation}\label{eq:stirling}
\Gamma(z)~=~ \sqrt{\frac{2\pi}{z}}\left(\frac z\eee\right)^z\left(1+O\left(\frac{1}{z}\right)\right),
\quad |z|\to+\infty,\quad
|\arg z| < \pi -\eps,
\end{equation}
where $\eps>0$ is arbitrary and a constant in the Landau symbol $O$ depends only on $\eps$.

Let $K$ be a compact subset of $\mathcal{D}^{\prime}_1$. Then there exists $\eps(K)$ such that $|\arg(zn)|<\pi-\eps(K)$ and $|\arg((z+1)n)|<\pi-\eps(K)$ for all $z\in K$ and all sufficiently large $n\in\N$. Thus,~\eqref{eq:stirling} yields
\begin{multline*}
\frac{S_n(z n)}{n!}~=~
\frac
{\sqrt{\frac{2\pi}{(z+1)n}} \left(\frac{(z+1) n}{\eee}\right)^{(z+1)n}}
{\sqrt{\frac{2\pi}{z n}} \left(\frac{z n}{\eee}\right)^{z n}
\sqrt{2\pi n} \left(\frac{n}{\eee}\right)^{n}}\left(1+O\left(\frac{1}{n}\right)\right)\\
=
\frac 1 {\sqrt{2\pi n}} \sqrt{\frac{z}{z+1}}\eee^{n ((z+1)\log (z+1) - z \log z)}\left(1+O\left(\frac{1}{n}\right)\right),\quad n\to\infty,
\end{multline*}
where a constant in the Landau symbol depends only on $\eps(K)$.
\end{proof}

\begin{proposition}\label{prop:asympt_factorial2}
There exists a domain $\mathcal{D}^{\prime}_2\subset \C$ containing $\R\setminus [-\eee,0]$ such that
\begin{equation}\label{eq:asympt_touchard_poly}
\frac{T_n(nz)}{n!}~=~
\frac 1 {\sqrt{2\pi n}} \frac {1} {\sqrt {W_0(1/z)+1}} \eee^{n L_2(z)}\left(1+O\left(\frac{1}{n}\right)\right),\quad n\to\infty,
\end{equation}
locally uniformly in $z\in\mathcal{D}^{\prime}_2$, where the function $L_2$ is defined by~\eqref{eq:def_L_2}.
\end{proposition}
\begin{proof}
Let us assume first that $z_0$ is a fixed real number and $z_0\in \R\setminus [-\eee,0]$. Using the first equality in~\eqref{eq:touchard_poly_def} and the Cauchy integral formula we can write
$$
\frac{T_n(nz_0)}{n!} =\frac{1}{2\pi \iii} \oint_{\gamma} \frac{1}{x}\left(\frac{\eee^{z_0 (\eee^x-1)}}{x}\right)^n \dd x,
$$
where $\gamma$ is an arbitrary circle centered at $0$ and oriented  counterclockwise. We are now going to apply the saddle point method in a form given in Theorem~\ref{thm:fedoryuk} below with
$$
\Omega_x:=\C\setminus\{0\},\quad \Omega_z:=\C,\quad g(x,z):=x^{-1},\quad f(x,z):=x^{-1}\eee^{z (\eee^x-1)},
$$
and $\gamma=\gamma(z_0)$ being a circle centered at $0$ and of a radius to be specified next.

Given $z_0$, we want $\gamma(z_0)$ to pass through a simple saddle point of $x\mapsto \log f(x,z_0)$. This saddle point $\zeta=\zeta(z_0)$ is defined by the equation $\zeta(z_0)\eee^{\zeta(z_0)} = 1/z_0$.  One of its solutions  is given by $\zeta(z_0)= W_0(1/z_0)\in (-1,\infty)\setminus\{0\}$, since we assume $z_0\in \R\setminus [-\eee,0]$. Note that
$$
\log f(\zeta(z_0),z_0)
=
z_0(\eee^{W_0(1/z_0)} - 1) - \log W_0(1/z_0)
=
W_0(1/z_0) + \frac{1}{W_0(1/z_0)} - z_0 + \log z_0
=
L_2(z_0),
$$
since $\eee^{W_0(1/z_0)}W_0(1/z_0) = 1/z_0$ and hence $-\log W_0(1/z_0) = \log z_0 + W_0(1/z_0)$.  Further, $\left(\log f(x,z)\right)_{xx}^{\prime\prime} = z\eee^x + x^{-2}$ and hence,
$$
\left(\log f(x,z_0)\right)_{xx}^{\prime\prime}|_{x=\zeta(z_0)}
=
z_0 \eee^{W_0(1/z_0)} + \frac{1}{W_0^2(1/z_0)}
=
\frac{1}{W_0(1/z_0)}+\frac{1}{W_0^2(1/z_0)}
=
\frac{W_0(1/z_0) + 1}{W_0^2(1/z_0)}
>0.
$$
Let $\gamma = \gamma(z_0)$ be the circle centered at the origin and passing through $\zeta(z_0)$. Taking everything together and applying formula~\eqref{eq:fedoryuk_asymp} in Theorem~\ref{thm:fedoryuk} proves that there exists $\delta(z_0)>0$ such that
\begin{equation}\label{eq:touchard_t_n_asymp1}
\frac{T_n(n z)}{n!}~=~\frac 1 {\sqrt{2\pi n}} \frac {1} {\sqrt {W_0(1/z)+1}} \eee^{n L_2(z)}\left(1+O\left(\frac{1}{n}\right)\right),\quad n\to\infty,
\end{equation}
uniformly in the disk $\mathbb{D}_{\delta(z_0)}(z_0)$. The only assumption of Theorem~\ref{thm:fedoryuk} which requires verification is the fact that $x\mapsto |x^{-1}\eee^{z_0 (\eee^x-1)}|$ attains a unique maximum on $\gamma(z_0)$ at $\zeta(z_0)$, or equivalently $x\mapsto \eee^{z_0\Re(\eee^x)}$
attains a unique maximum on $\gamma(z_0)$ at $\zeta(z_0)$. This is obvious if $z_0>0$ since the Taylor expansion of $x\mapsto \eee^{z_0\Re(\eee^x)}$ has only nonnegative coefficients and $\zeta(z_0)>0$. If $z_0<-\eee$, then $\zeta(z_0)\in (-1,0)$. Thus, it suffices to show that the function $x\mapsto \Re(\eee^x)=\eee^{\Re(x)}\cos(\Im(x))$ attains on a circle of radius smaller than $1$ centered at the origin a unique minimum at a negative real. This can be easily checked by solving the constrained optimization problem:
$$
\eee^a\cos(b)\longrightarrow \min\quad \text{ subject to }\quad a^2+b^2=R^2<1,\quad R \text{ is fixed},\quad  a,b\in\R.
$$

Put $\mathcal{D}^{\prime}_2:=\bigcup_{z_0\in \R\setminus [-\eee,0]}\mathbb{D}_{\delta(z_0)}(z_0)$. If $K$ is a compact subset of $\mathcal{D}^{\prime}_2$, then $K$ can be covered by finitely many disks from the collection $\{\bD_{\delta(z_0)}(z_0):z_0\in \R\setminus [-\eee,0]\}$ which yields the desired uniformity, that is, \eqref{eq:touchard_t_n_asymp1} holds uniformly in $z\in K$.
\end{proof}

Using Propositions~\ref{prop:asympt_factorial1} and~\ref{prop:asympt_factorial2} one can easily deduce Theorems~\ref{thm:mod_phi_stirling1} and~\ref{thm:mod_phi_stirling2}. The proof of Theorem~\ref{thm:mod_phi_stirling3} is more involved but still relies on the saddle point method.

\begin{proof}[Proof of Theorem~\ref{thm:mod_phi_stirling1}]
Applying~\eqref{eq:asympt_factorial} twice, we obtain
\begin{align*}
\frac{S_n(\vartheta n)}{n!}&=\frac 1 {\sqrt{2\pi n}} \sqrt{\frac{\vartheta}{\vartheta + 1}} \eee^{n L_1(\vartheta)}\left(1+O\left(\frac{1}{n}\right)\right),\\
\frac{S_n(\vartheta \eee^z n)}{n!}&= \frac 1 {\sqrt{2\pi n}} \sqrt{\frac{\vartheta \eee^z}{\vartheta \eee^z+1}} \eee^{n L_1(\vartheta \eee^z)}\left(1+O\left(\frac{1}{n}\right)\right),
\end{align*}
provided $\vartheta\in \mathcal D^{\prime}_1$ and $\vartheta \eee^z\in\mathcal D^{\prime}_1$. The first condition is fulfilled automatically, whereas the second hold if $z\in \mathcal{D}_1$, where the domain $\mathcal{D}_1$ is defined as the preimage of $\mathcal{D}^{\prime}_1$ under the map $z\mapsto \vartheta \eee^z$. Clearly, $\mathcal{D}_1$ contains the real axis.
Since $\E \eee^{zX^{(1)}_{n,\vartheta n}}$ is given by the quotient of these expressions, we arrive at~\eqref{eq:mod_phi_1}.
\end{proof}
\begin{proof}[Proof of Theorem~\ref{thm:mod_phi_stirling2}]
The proof of~\eqref{eq:mod_phi_2} is similar. Applying\eqref{eq:asympt_touchard_poly} twice, we obtain
\begin{align*}
\frac{T_n(\vartheta n)}{n!} &= \frac 1 {\sqrt{2\pi n}} \frac {\eee^{n L_2(\vartheta)}} {\sqrt {W_0(\vartheta^{-1})+1}}\left(1+O\left(\frac{1}{n}\right)\right),\\
\frac{T_n(\vartheta \eee^z n)}{n!} &= \frac 1 {\sqrt{2\pi n}} \frac {\eee^{n L_2(\vartheta \eee^z)}} {\sqrt {W_0(\vartheta^{-1}\eee^{-z})+1}}\left(1+O\left(\frac{1}{n}\right)\right).
\end{align*}
Since $\E \eee^{zX^{(2)}_{n,\vartheta n}}$ is given by the quotient of these expressions, we arrive at~\eqref{eq:mod_phi_2}. The set $\mathcal{D}_2$ is defined as the preimage of $\mathcal{D}_2^{\prime}$ under the map $z\mapsto \vartheta \eee^z$. Clearly, $\R\subset \mathcal{D}_2$.
\end{proof}

The next lemma provides a useful integral representation of the generating function $\mathcal{G}_{n,\theta}^{(3)}$ of $X^{(3)}_{n,\theta}$.

\begin{lemma}\label{lem:stirling_3_integral_rep}
Assume that $n\in\mathbb{N}$ and $\theta>0$. Let $\mathcal{D}$ be an arbitrary bounded domain in $\mathbb{C}$. Then for an arbitrary sufficiently small closed contour $\gamma$ encircling the origin counterclockwise  the following holds true:
$$
\mathcal{G}_{n,\theta}^{(3)}(t) = \frac{n!}{\theta^n 2\pi\iii}\oint_{\gamma}s^{-n-1}\left(1+t(\eee^s-1)\right)^{\theta}{\rm d}s
,\quad t\in \mathcal{D}.
$$
\end{lemma}
\begin{proof}
Let $\gamma$ be an arbitrary closed contour encircling the origin such that
$$
|t(\eee^s-1)|<1/2,\qquad \text{ for all } s\in\gamma \text{ and } t\in \mathcal{D}.
$$
It is well-known that
$$
\frac{(\eee^z-1)^k}{k!}=\sum_{n=k}^{\infty}\stirlingsec{n}{k}\frac{z^n}{n!},\quad z\in\mathbb{C},\quad k\in\mathbb{N},
$$
and, thereupon,
$$
\stirlingsec{n}{k}=\frac{n!}{k!}[z^n](\eee^z-1)^k=\frac{n!}{k!2\pi\iii}\oint_{\gamma}\frac{(\eee^s-1)^k}{s^{n+1}}{\rm d}s.
$$
Note that this formula holds also for $k>n$ and $k=0$ yielding the obvious equality $\stirlingsec{n}{k}=0$. Thus,
\begin{multline*}
\mathcal{G}_{n,\theta}^{(3)}(t)=\sum_{k=1}^{n} \P[X^{(3)}_{n,\theta }=k]t^k=\theta^{-n}\sum_{k=1}^{n} \stirlingsec{n}{k}(\theta)^{\underline{k}}t^k=\frac{n!}{\theta^n2\pi\iii}\sum_{k=0}^{\infty}(\theta)^{\underline{k}}t^k\frac{1}{k!}\oint_{\gamma}\frac{(\eee^s-1)^k}{s^{n+1}}{\rm d}s\\
=\frac{n!}{\theta^n2\pi\iii}\oint_{\gamma}s^{-n-1}\sum_{k=0}^{\infty}\binom{\theta}{k}(t(\eee^s-1))^k{\rm d}s=\frac{n!}{\theta^n2\pi\iii}\oint_{\gamma}s^{-n-1}(1+t(\eee^s-1))^{\theta}{\rm d}s.
\end{multline*}
\end{proof}

\begin{proof}[Proof of Theorem~\ref{thm:mod_phi_stirling3}]
From Lemma~\ref{lem:stirling_3_integral_rep} it follows that
\begin{equation}\label{eq:moment_gener_X_n_3}	
\E \eee^{z X^{(3)}_{n,\vartheta n }}
=
\frac{n!}{(n\vartheta)^n 2\pi\iii}\oint_{\gamma}x^{-n-1}\left(1+\eee^z(\eee^x-1)\right)^{\vartheta n}{\rm d}x
,\quad z\in\C,
\end{equation}
where $\gamma$ is a sufficiently small closed contour encircling the origin. We need to estimate the  integral
%According to formula~\eqref{eq:moment_gener_X_n_3} we need to estimate the  integral
$$
\int_{\gamma}\frac{1}{x}\left(\frac{(1+\eee^z(\eee^x-1))^{\vartheta}}{x}\right)^n{\rm d}x
$$
using the saddle point method. This is done again by appealing to Theorem~\ref{thm:fedoryuk} with
$$
\Omega_x:=\C\setminus\{0\},\quad \Omega_z:=\C,\quad g(x,z):=\frac{1}{x},\quad f(x,z):=\frac{(1+\eee^z(\eee^x-1))^{\vartheta}}{x}.
$$
Fix $z_0\in\R$. The saddle points of the function $x\mapsto \log f(x,z_0)$ are the solutions to the equation
\begin{equation}\label{eq:x_3_saddle_point_equation_init}
\vartheta x \eee^{x+z_0}=1+\eee^{z_0}(\eee^{x}-1),
\end{equation}
which can also be written as
\begin{equation}\label{eq:x_3_saddle_point_equation}
\left(x-\frac{1}{\vartheta}\right)\eee^{x-1/\vartheta}=\vartheta^{-1}(\eee^{-z_0}-1)\eee^{-1/\vartheta}.
\end{equation}
Note that $1-\vartheta \eee^{1/\vartheta-1}\leq 0$, for all $\vartheta>0$. Thus,
$$
\vartheta^{-1}(\eee^{-z_0}-1)\eee^{-1/\vartheta}>-1/\eee,
$$
and one solution to~\eqref{eq:x_3_saddle_point_equation} is given by
$$
\zeta(z_0)=1/\vartheta+W_0(\vartheta^{-1}(\eee^{-z_0}-1)\eee^{-1/\vartheta})=L_3(z_0;\vartheta).
$$
Since $W_0$ is strictly increasing on $[-1/eee,+\infty)$,
$$
\zeta(z_0)> 1/\vartheta+W_0(-\vartheta^{-1}\eee^{-1/\vartheta})=0.
$$
From~\eqref{eq:x_3_saddle_point_equation_init} it is also clear that $f(\zeta(z_0),z_0)\neq 0$. In view of
$$
\left(\log f(x,z)\right)_{xx}^{\prime\prime}=\frac{\vartheta \eee^{x}\eee^{z}(1-\eee^z)}{(1+\eee^{z}(\eee^x-1))^2}+\frac{1}{x^2},
$$
and using~\eqref{eq:x_3_saddle_point_equation_init}, we obtain
$$
\left(\log f(x,z)\right)_{xx}^{\prime\prime}\Big|_{x=\zeta(z_0)}=\frac{\vartheta \zeta(z_0)+\vartheta-1}{\vartheta \zeta^2(z_0)}>0,
$$
where the inequality is a consequence of $W_0(u)>-1$, for all $u>-1/\eee$.

Let $\gamma=\gamma(z_0)$ be the circle centered at the origin and having the positive radius $\zeta(z_0)$. Let us show that the function $x\mapsto |f(x,z_0)|$ attains a unique maximum on $\gamma$ at $x=\zeta(z_0)$. Since, for $x\in\gamma(z_0)$, it holds $|f(x,z_0)|=|\zeta(z_0)|^{-1} |1+\eee^{z_0}(\eee^x-1)|^{\vartheta}$, and the function $x\mapsto 1+\eee^{z_0}(\eee^x-1)$ has positive coefficients in the Taylor expansion around the origin, the claim follows. Summarizing, by Theorem~\ref{thm:fedoryuk}, there exists $\delta(z_0)>0$ such that
\begin{align*}
\int_{\gamma}\frac{1}{x}\left(\frac{(1+\eee^z(\eee^x-1))^{\vartheta}}{x}\right)^n{\rm d}x&=(f(\zeta(z),z))^n\sqrt{-\frac{2\pi\vartheta\zeta^2(z)}{n(\vartheta\zeta(z)+\vartheta-1)}}\frac{1}{\zeta(z)}\left(1+O\left(\frac{1}{n}\right)\right)\\
&=\left(\frac{(\vartheta \zeta(z)\eee^{z+\zeta(z)})^{\vartheta}}{\zeta(z)}\right)^n\sqrt{-\frac{2\pi\vartheta}{n(\vartheta\zeta(z)+\vartheta-1)}}\left(1+O\left(\frac{1}{n}\right)\right).
\end{align*}
uniformly in the disk $\mathbb{D}_{\delta(z_0)}(z_0)$. Plugging this into~\eqref{eq:moment_gener_X_n_3} and using the Stirling approximation for the factorial we arrive at
\begin{align*}
\E \eee^{z X_{n,\vartheta n}^{(3)}}&=\left(\frac{(\vartheta \zeta(z)\eee^{z+\zeta(z)})^{\vartheta}}{\zeta(z)\vartheta \eee}\right)^n\sqrt{\frac{\vartheta}{\vartheta\zeta(z)+\vartheta-1}}\left(1+O\left(\frac{1}{n}\right)\right)\\
&=\eee^{n\varphi_3(z;\vartheta)}\sqrt{\frac{\vartheta}{\vartheta L_3(z;\vartheta)+\vartheta-1)}}\left(1+O\left(\frac{1}{n}\right)\right).
\end{align*}
Put $\mathcal{D}_3:=\bigcup_{z_0\in\R}\mathbb{D}_{\delta(z_0)}(z_0)$. The same compactness argument, as we have used in the proof of Proposition~\ref{prop:asympt_factorial2}, shows local uniformity on $\mathcal{D}_3$. The proof is complete.
\end{proof}
\begin{remark}\label{rem:uniformity_in_theta}
The claims of Theorems~\ref{thm:mod_phi_stirling1},~\ref{thm:mod_phi_stirling2} and~\ref{thm:mod_phi_stirling3} hold also locally uniformly in $\vartheta\in (0,\infty)$ in the following sense. For every compact set $K\subset (0,\infty)$, there exist domains $\mathcal{D}_{i}=\mathcal{D}_{i}(K)$, $i=1,2,3$, such that~\eqref{eq:mod_phi_1},~\eqref{eq:mod_phi_2} and~\eqref{eq:mod_phi_3} hold uniformly in $\vartheta\in K$ and $z$ in compact subsets of $\mathcal{D}_{i}=\mathcal{D}_{i}(K)$, $i=1,2,3$, respectively. %For $i=1,2$ this follows immediately from the proofs.
\end{remark}

\subsection{Proof of Theorem~\ref{theo:local_limit2}}\label{sec:local_limit}

As we have already mentioned, Theorem~\ref{theo:local_limit2} follows from Theorem 2.7 in~\cite{kabluchko_marynych_sulzbach} once we check, see Eq.~(11) in~\cite{kabluchko_marynych_sulzbach}, that for every compact set $K\subset \R$, $a\in (0,\pi)$, $\vartheta>0$ and $i=1,2,3$,
\begin{equation}\label{eq:A4_condition_general}
\sup_{t\in K}\left|\eee^{-\varphi_i(t;\vartheta)n}\int_{a}^{\pi}\left|\E \eee^{(t+\iii u)X_{n,\vartheta n}^{(i)}}\right|{\rm d}u\right|=o(n^{-1}),\quad n\to\infty.
\end{equation}
Below we shall treat in details the case $i=2$ and then explain how to check~\eqref{eq:A4_condition_general} in two other cases using similar arguments.

Let us show that there exists $\delta\in (0,1)$ and $n_0\in\N$, such that
\begin{equation}\label{eq:A4_condition_proof1}
\sup_{t\in K}\sup_{a\leq u\leq \pi} \frac{|T_n(n\vartheta \eee^{t+\iii u})|}{T_n(n\vartheta \eee^t)}\leq (1-\delta)^n,\quad n\geq n_0.
\end{equation}
This is sufficient for our purposes because
$$
K\ni t\mapsto \eee^{-\varphi_i(t;\vartheta)n} \frac{T_n(n\vartheta \eee^t)}{T_n(n\vartheta)},
$$
converges uniformly, as $n\to\infty$, to a bounded function by~\eqref{eq:mod_phi_2}.

Let $-z_{1,n},\ldots,-z_{n,n}$ denote zeros of $x\mapsto T_n(nx)$ which are all real and negative. Let $Z_n$ be the subset of zeros of $x\mapsto T_n(nx)$ lying in $[-1,-1/2]$. From Elbert's result~\eqref{eq:elbert_measures}-\eqref{eq:elbert_limit} we know that there exist $n_0\in\N$ and $c_1\in (0,1)$ such that $|\mathcal{Z}_n|> c_1 n$, for all $n\geq n_0$. Write
\begin{equation}\label{eq:A4_condition_proof2}
\frac{|T_n(n\vartheta \eee^{t+\iii u})|}{T_n(n\vartheta \eee^t)}=\prod_{k=1}^{n}\frac{|n(\vartheta \eee^{t+\iii u}+z_{k,n})|}{n(\vartheta \eee^{t}+z_{k,n})}
\leq
\prod_{k:\;-z_{k,n}\in\mathcal{Z}_n}\frac{|n(\vartheta \eee^{t+\iii u}+z_{k,n})|}{n(\vartheta \eee^{t}+z_{k,n})}=
\prod_{k:\;z_{k,n}\in [1/2,1]}\frac{|n(\vartheta \eee^{t+\iii u}+z_{k,n})|}{n(\vartheta \eee^{t}+z_{k,n})},
\end{equation}
where for the inequality we used that $|n(\vartheta \eee^{t+\iii u}+z_{k,n})|\leq n(\vartheta \eee^{t}+|z_{k,n}|)=n(\vartheta \eee^{t}+z_{k,n})$, since $z_{k,n}>0$. Let us show now that, for every $z\in [1/2,1]$,
\begin{equation}\label{eq:A4_condition_proof3}
\sup_{t\in K}\sup_{u\in [a,\pi]}\frac{|\vartheta \eee^{t+\iii u}+z|^2}{(\vartheta \eee^{t}+z)^2}\leq 1-\delta_1,
\end{equation}
for some $\delta_1\in (0,1)$. This estimate together with \eqref{eq:A4_condition_proof2} implies~\eqref{eq:A4_condition_proof1} with $\delta:=1-(1-\delta_1)^{c_1/2}$. After some elementary calculations one can see that~\eqref{eq:A4_condition_proof3} is equivalent to
$$
(z\vartheta \eee^t)^{-1}((\vartheta\eee^t)+z^2)\leq 2\delta_1^{-1}(1-\delta_1-\cos(u)),\quad t\in K,\quad u\in [a,\pi],\quad z\in [1/2,1].
$$
The latter clearly holds for a sufficiently small $\delta_1 \in (0,1-\cos(a))$ because the left-hand side is uniformly bounded for $t\in K$ and $z\in [1/2,1]$, whereas the right-hand side can be made arbitrarily large. This completes the proof in case $i=2$.

\vspace{2mm}

\noindent
{\sc Case $i=1$.} The proof proceeds in the same way
as in case $i=2$ by using that the zeros of $S_n(-nx)$, that is, the set $\{0,1/n,2/n,\ldots,(n-1)/n\}$, are all nonnegative, and the number of zeros of $x\mapsto S_n(nx)$ in $[-1,-1/2]$ grows linearly as $n\to\infty$.

\noindent
{\sc Case $i=3$.} The proof proceeds in the same way
as in case $i=2$ by using that the zeros of $\mathcal{G}_{n,\vartheta n}^{(3)}$ are all nonpositive, and their number grows linearly in an arbitrary subset of $[0,m_\vartheta]$ having positive Lebesgue measure, as $n\to\infty$. The latter claim is secured by Proposition~\ref{prop:convergence_of_zeros_of_x_3}.

\section{Proofs for Section~\ref{sec:finite_free}}\label{sec:finite_free_proof}

\begin{proof}[Proof of Proposition~\ref{prop:zeros}]
Since the function $\Psi$ does not vanish on $\mathcal{D}\cap \R$, there exists a domain $\widehat{\mathcal{D}}$ such that $\mathcal{D}\cap \R\subset \widehat{\mathcal{D}}$ and $\Psi$ does not vanish on $\widehat{\mathcal{D}}$. Thus,~\eqref{eq:mod_phi_def_conv} entails
$$
\lim_{n\to\infty}\frac{1}{w_n}\log \E \eee^{z X_n}~=~\varphi(z),
$$
locally uniformly on $\widehat{\mathcal{D}}$, and
$$
\lim_{n\to\infty}\frac{1}{w_n}\log \mathcal{G}_n(t)~=~\varphi(\log t),
$$
locally uniformly on
$\exp(\widehat{\mathcal{D}}):=\{\eee^z:z\in\widehat{\mathcal{D}}\}$.
Since the locally uniform convergence of analytic functions implies
locally uniform convergence of their derivatives, we obtain
\begin{equation}\label{eq:mod_phi_to_zeros_proof1}
\lim_{n\to\infty}\frac{1}{w_n}\sum_{x\in \Zeros(\mathcal G_n)}\frac{1}{t-x}~=~\frac{\varphi'(\log t)}{t},\quad t\in\exp(\widehat{\mathcal{D}}).
\end{equation}
Note that
$$
\frac{1}{w_n}\sum_{x\in \Zeros(\mathcal G_n)}\frac{1}{t-x}=\int_{[0,\infty)}\frac{\mathcal{Z}_n(\dd x)}{t+x},\quad t\in\C\setminus (-\infty,0],
$$
and the right-hand side is equal to the negative of the Stieltjes transform of $\mathcal{Z}_n$ evaluated at $-t$. Therefore, equation~\eqref{eq:mod_phi_to_zeros_proof1} tells us that the Stieltjes transform of $\mathcal{Z}_n$ converges locally uniformly on an open subset of $\C$ to a limit, as $n\to\infty$. Thus, by part (vi) of Proposition 2.1.2 in \cite{pastur_shcherbina_book} (where the condition that $\mathcal{Z}_n(\R) = {\rm deg}(\mathcal G_n)/w_n$ stays bounded is required but not stated explicitly), $\mathcal{Z}_n$ converges vaguely to some finite measure $\mathcal{Z}$, as $n\to\infty$. Since for every $t\in\C\setminus [0,+\infty)$, the function $t\mapsto 1/(t-x)$ is continuous on $[0,+\infty)$ and vanishes at $+\infty$, it follows that the Stieltjes transform of $\mathcal Z_n$ converges to that of $\mathcal Z$ pointwise on $\C\setminus [0,+\infty)$. Since the latter is equal to $\varphi'(\log (-t))/t$ for all $-t\in \exp(\widehat{\mathcal{D}})$, the function $\varphi'(\log (-t))/t$ admits an analytic continuation to $\C\setminus [0,+\infty)$ and is equal to the Stieltjes transform of $\mathcal Z$ there.
\end{proof}

\begin{proof}[Proof of Proposition~\ref{prop:density_properties}]
Recall from Proposition~\ref{prop:convergence_of_zeros_of_x_3} that the Stieltjes transform of $\mathcal{Z}^{\vartheta}$ is given by
\begin{equation}\label{eq:formula_st_x_3_measure_explicit}
\int_{[0,\infty)}\frac{\mathcal{Z}^{\vartheta}(\dd x)}{z-x}~=~
\frac{\vartheta}{z} -\frac{\vartheta^2}{z(z+1)} \frac{W_0\left(\frac {- z^{-1} - 1}{\vartheta \eee^{1/\vartheta}}\right)}{1 + \vartheta W_0\left(\frac {- z^{-1} - 1}{\vartheta \eee^{1/\vartheta}}\right)},
\qquad z\in \C\setminus [0,+\infty).
\end{equation}
The total mass of $\mathcal{Z}^{\vartheta}$ can be calculated as follows:
$$
\mathcal{Z}^{\vartheta}(\R)=\lim_{s\to-\infty}s\int_{[0,\infty)}\frac{\mathcal{Z}^{\vartheta}(\dd x)}{s-x}=
\begin{cases}
1,&\quad \vartheta\geq 1,\\
\vartheta, &\quad \vartheta\in (0,1).
\end{cases}
$$
In the case $\vartheta\in (0,1)$ this is trivial since the ratio $W_0/(1+\vartheta W_0)$ on the right-hand side of~\eqref{eq:formula_st_x_3_measure_explicit} is bounded, whereas, for $\vartheta>1$ the formula follows from Taylor's expansion
$$
W_0\left(\frac {- s^{-1} - 1}{\vartheta \eee^{1/\vartheta}}\right)=W_0\left(-\frac{1}{\vartheta}\eee^{-1/\vartheta}\right)-\frac{1}{s\vartheta \eee^{1/\vartheta}}W_0^{\prime}\left(-\frac{1}{\vartheta}\eee^{-1/\vartheta}\right)+O\left(\frac{1}{s^2}\right)\overset{\eqref{eq:W_0_der}}{=}-\frac{1}{\vartheta}-\frac{1}{\vartheta-1}\frac{1}{s}+O\left(\frac{1}{s^2}\right),\quad s\to -\infty.
$$
Finally, the case $\vartheta=1$ can be treated using the first  expansion in~\eqref{eq:asympt_W0_at_-1over_e} with $\delta=-1/(s\eee)$ which yields $\int_{[0,\infty)}\frac{\mathcal{Z}^{\vartheta}(\dd x)}{s-x}=1/s+O(|s|^{-3/2})$, as $s\to-\infty$.

The formula for the density given in~\eqref{eq:density_alloc_explicit} follows from the Stieltjes--Perron inversion formula; see~\cite[Proposition 2.1.2 on p.~35]{pastur_shcherbina_book} and~\cite[pp.~124--125]{akhiezer_book}. More precisely, part (vii) of~\cite[Proposition 2.1.2]{pastur_shcherbina_book}  yields the explicit formula~\eqref{eq:density_alloc_explicit} for the density on $(0,\infty)$ (and the existence of this density), whereas part~(v) implies the absence of the atom at $0$ taking into account the asymptotics~\eqref{eq:density_alloc_asympt_0} as $t\downarrow 0$, which we shall prove below.

To prove that the density vanishes outside $(0,m_{\vartheta})$ for $\vartheta\neq 1$,  recall that $W_0$ has a branch cut along $(-\infty, -1/\eee]$, while on $(-1/\eee,+\infty)$ it stays real. For $\vartheta>0$ we have $\vartheta \eee^{1/\vartheta}\geq \eee$ with equality iff $\vartheta = 1$. For $\vartheta \neq 1$ the following conditions on $t>0$ are equivalent:
$$
\frac {- t^{-1} - 1}{\vartheta \eee^{1/\vartheta}} > - \frac 1 \eee
\quad
\Leftrightarrow
\quad
t > \frac 1 {\vartheta \eee^{(1/\vartheta) - 1} - 1}= m_\vartheta.
$$
Hence, for $\vartheta \neq 1$ and  $t>m_\vartheta > 0$, the imaginary part on the right-hand side of~\eqref{eq:density_alloc_explicit} vanishes.

To prove~\eqref{eq:density_alloc_asympt_0}, we rely on the asymptotics
\begin{equation}\label{eq:asympt_lambert_rep}
W_0(-R+ \iii 0) = \log (R/\log R) + \pi \iii + o(1),
\qquad
R\to+\infty.
\end{equation}
To derive~\eqref{eq:asympt_lambert_rep}, we put $W_0(-R+ \iii 0) = \log (R/\log R) + \pi \iii + \delta(R)$ with an unknown $\delta(R)$. Then,
\begin{equation}\label{eq:delta(R)}
\left(1  + \frac{\pi \iii -\log \log R}{\log R} + \frac{\delta(R)} {\log R}\right) \eee^{\delta(R)} = 1.
\end{equation}
It is known from~\cite[Lemma~2.3~(d)]{elbert1} that $\Im \delta(R)\to 0$ as $R\to+\infty$. If, along some subsequence of $R$'s diverging to $+\infty$, $\Re \delta(R)/\log R$ stays bounded away from $0$, then the limit of the absolute value of the left-hand side of~\eqref{eq:delta(R)} can not be equal to $1$, which is a contradiction. Hence, $\delta(R)/\log R$ goes to $0$ and it follows from~\eqref{eq:delta(R)} that, in fact, $\delta(R) \to 0$ as $R\to +\infty$,
%and in fact,
%$$
%delta(R) = \frac{\pi \iii -\log \log R}{\log R} + o\left(\frac{1}{\log R}\right)
%$$
%as $R\to+\infty$,
thus proving~\eqref{eq:asympt_lambert_rep}.

Writing $R:= \frac {t^{-1} + 1}{\vartheta \eee^{1/\vartheta}}$ we observe that
$$
\Im \frac{W_0(-R + \iii 0)}{1 + \vartheta W_0(-R + \iii 0)}
\sim \frac{\pi}{\vartheta^2 (\log R)^2},
\qquad R\to +\infty.
$$
Plugging this into~\eqref{eq:density_alloc_explicit} and noting that $\log R \sim |\log t|$ completes the proof of~\eqref{eq:density_alloc_asympt_0}.

To prove~\eqref{eq:density_alloc_asympt_infty_vartheta1}, we apply the second equality in~\eqref{eq:asympt_W0_at_-1over_e} with $\delta:= 1/(t\eee)$. This yields
$$
\frac{W_0\left(\frac {- t^{-1} - 1 + \iii 0}{\eee}\right)}{1 + W_0\left(\frac {- t^{-1} - 1 + \iii 0}{\eee}\right)}
\sim
\frac{\iii}{\sqrt{2/t}},
\qquad t\to +\infty,
$$
and~\eqref{eq:density_alloc_asympt_infty_vartheta1} follows by plugging this into~\eqref{eq:density_alloc_explicit}.

Let us prove~\eqref{eq:density_alloc_asympt_infty_vartheta_neq1}.  For $t= m_\vartheta - \eps$, the argument of the Lambert function in~\eqref{eq:density_alloc_explicit} satisfies
$$
\frac {- t^{-1} - 1 + \iii 0}{\vartheta \eee^{1/\vartheta}} = -\frac 1 \eee - \frac{\eps}{m_\vartheta^2 \vartheta \eee^{1/\vartheta}} +\iii 0, \qquad \eps \downarrow 0.
$$
The claim follows by applying the second formula in~\eqref{eq:asympt_W0_at_-1over_e} with $\delta = \frac{\eps}{m_\vartheta^2 \vartheta \eee^{1/\vartheta}}$ and plugging the result into~\eqref{eq:density_alloc_explicit}.
\end{proof}

\section{Appendix}

\subsection{The Lambert \texorpdfstring{$W$}{W}-function and its principal branch \texorpdfstring{$W_0$}{W0}.}\label{appendix_Lambert}

%\begin{figure}[t]
%\begin{center}
%\includegraphics[width=0.49\textwidth]{Lambert_region.pdf}
%\end{center}
%\caption
%{
%The image of the slitted plane $\C\setminus [-1/\eee,\infty)$ under the principal branch $W_0$. The curve bounding the region is given by %$\gamma= \{z\in \C:  z \eee^{z} \in (-\infty, -1/\eee], |\Im z|<\pi\}$. The part of $\gamma$ that is located in the upper, respectively, lower, %half-plane is the image of $(-\infty,-1/\eee) + \iii 0$, respectively, $(-\infty,-1/\eee) - \iii 0$, under $W_0$.
%}
%\label{fig:lambert_region}
%\end{figure}

The \emph{Lambert $W$-function} is a multivalued analytic function defined by the implicit equation
$$
W(z) \eee^{W(z)} = z, \qquad z\in\C.
$$
It has a branch point at $z=-1/\eee$ and infinitely many branches whose structure was discussed in detail in~\cite{corless_etal}.
We mostly need the principal branch $W_0$ which is defined on the whole complex plane with a branch cut at $(-\infty, -1/\eee]$. On the real line, the function $w\mapsto w\eee^w=z$ has a unique minimum $-1/\eee$ attained at $w=-1$. Thus, there is a well-defined inverse function $W_0:(-1/\eee , \infty) \to (-1,\infty)$, called the principal branch, which is monotone increasing, and satisfies
$$
W_0(z) \eee^{W_0(z)} = z, \qquad \lim_{x\downarrow -1/\eee} W_0(x) = -1, \qquad W_0(0)=0, \qquad  \lim_{x\uparrow +\infty} W_0(x) = +\infty.
$$
Moreover, it is possible to extend $W_0$ analytically to the slitted complex plane $\C\setminus (-\infty, -1/\eee]$.  The principal branch maps $\C\setminus (-\infty, -1/\eee]$ conformally to the region $\{a+b \iii : a > -b  \cot b, b\in (-\pi, \pi)\}$ (where $0\cot 0:=1$); see~\cite[Lemma~2.3 (ii)]{elbert1}.

We need the limit values of $W_0(z)$ when the complex variable $z$ approaches the branch cut $(-\infty, -1/\eee]$. Such limit value depends on whether $z$ stays in the upper half-plane or in the lower half-plane. We can define
$$
W_0(x + \iii 0) :=  \lim_{\eps\downarrow 0} W_0(x + \iii \eps),
\qquad
W_0(x - \iii 0) :=  \lim_{\eps\downarrow 0} W_0(x - \iii \eps),
\qquad
x\in (-\infty, -1/\eee).
$$
These limit values are complex conjugate to each other:
$$
W_0(x - \iii 0) = \overline {W_0(x + \iii 0)},
\qquad
x\in (-\infty, -1/\eee).
$$
The properties of $W_0(x \pm \iii 0)$ are summarized in~\cite[Lemma~2.3]{elbert1}.
In particular, it is known  that $x\mapsto \Im W_0(x + \iii 0)$ is a decreasing function of  $x\in (-\infty, -1/\eee)$ and that
$$
\lim_{x\downarrow -\infty} \Im W_0(x + \iii 0) = \pi,
\qquad
\lim_{x\uparrow -1/\eee} \Im W_0(x + \iii 0) = 0.
$$
We need the following formula for the derivative of the Lambert function:
\begin{equation}\label{eq:W_0_der}
W_0'(z) = \frac{W_0(z)}{z(1+W_0(z))}, \qquad z \in \C\setminus (-\infty,-1/\eee], \quad z\neq 0.
\end{equation}
In the first equation in formula~\eqref{eq:LDP_inverses} we have also encountered the function $W_{-1}(z)$. For  real $z\in (-1/\eee,0)$ it is defined as the unique solution $w$ to the equation $w\eee^{w}=z$ lying in $(-\infty,-1)$.

The Puiseux series of $W_0$ near the branch point $-1/\eee$ looks as follows, see~\cite[Equation~(4.22)]{corless_etal},
\begin{equation}\label{eq:asympt_W0_at_-1over_e}
W_0\left(-\frac 1\eee + \delta \right) = -1 + \sqrt{2\eee \delta} + O(\delta),\quad W_0\left(-\frac 1\eee - \delta \pm \iii 0\right) = -1 \pm\iii \sqrt{2\eee \delta} + O(\delta),
\qquad \delta\downarrow 0.
\end{equation}

\subsection{Uniformity in the saddle point method.}\label{sec:saddle_point}

While checking mod-$\varphi$ convergence for a sequence of random variables it is crucial to check that the limit relation
~\eqref{eq:mod_phi_def_conv} holds locally uniformly (in variable $z$) on a suitable domain $\mathcal{D}$ containing an interval of the real line.
The key tool in our asymptotic analysis is the saddle point method and, therefore, we need an appropriate result which ensures that the saddle point asymptotic expansion is locally uniform over $z\in\mathcal{D}$. General results of this type suitable for our needs can be found in~\cite[Chapter IV, \S 4]{fedoryuk_book}. For the ease of reference we provide below the corresponding theorem adopted to our settings.

For a Laplace-type integral
$$
I_n(z) = \int_{\gamma} (f(x,z))^n g(x,z) {\rm d} x
$$
we assume that:
\begin{itemize}
\item [(A1)] The functions $f(x,z)$ and $g(x,z)$ are analytic in a domain $\Omega_x\times \Omega_{z}\subset \C^2$, where $\Omega_x$ and $\Omega_{z}$ are domains in $\C$,
\item[(A2)] $\gamma: [0,1] \to \Omega_x$ is a piecewise smooth curve contained in $\Omega_x$ which has no self-intersections except, possibly, $\gamma(0)= \gamma(1)$.
\item [(A3)] For some $z_0 \in \Omega_{z}$, the contour $\gamma$ is a \emph{saddle point contour} meaning that the following conditions are satisfied:
\begin{itemize}[leftmargin=0cm]
\item[(i)] the function $x\mapsto \log f(x,z_0)$ has a simple saddle point $\zeta(z_0)$ in the relative interior of $\gamma$, that is,
$$
f(\zeta(z_0), z_0) \neq 0,
\quad
\frac{\partial \log f(x,z_0)}{\partial x}\Big|_{x=\zeta(z_0)}=0,
\quad
\frac{\partial^2 \log f(x,z_0)}{\partial x^2}\Big|_{x=\zeta(z_0)}\neq 0.
$$
Note  that although the logarithm is defined up to a summand $2\pi \iii n$, $n\in \Z$, only, both derivatives are well-defined.
\item[(ii)]
The function $x\mapsto |f(x,z_0)|$ has a unique maximum on $\gamma$ attained at $x = \zeta(z_0)$.
\item[(iii)] $g(\zeta(z_0), z_0) \neq 0$.
\item[(iv)]  In a small disk around $\zeta(z_0)$, the sublevel set $\{x\in \Omega_x: |f(x,z_0)| <  |f(\zeta(z_0), z_0)|\}$  consists of two sectors. The contour $\gamma$ passes through both of these sectors.
\end{itemize}
\end{itemize}

The next standard result is a uniform in $z$ expansion of $I_n(z)$, when  $n\to +\infty$.
\begin{theorem}\label{thm:fedoryuk}
Under the assumptions (A1), (A2), (A3) there exists $\delta>0$ such that, for every $z$ in the open disk $\bD_{\delta}(z_0):=\{z\in\C: |z-z_0|<\delta\}$, the function $x\mapsto \log f(x,z)$ has a unique saddle point $\zeta(z)$ satisfying $|\zeta(z) - \zeta(z_0)| <\delta$. Moreover, uniformly in the open disk $\bD_{\delta}(z_0)$ the following asymptotics holds true:
\begin{equation}\label{eq:fedoryuk_asymp}
I_n(z)
=
(f(\zeta(z),z))^{n} \cdot \sqrt{-\frac{2\pi}{n \, (\log f)''_{xx}(\zeta(z),z)}}
\cdot
(g(\zeta(z),z)+O(1/n)),
\quad
n\to+\infty.
\end{equation}
The branch of the square root is chosen such that ${\rm arg}(\sqrt{-(\log f)''_{xx}(\zeta(z_0),z_0)})$ is equal to the angle between the positive direction of the tangent line to $\gamma$ at $\zeta(z_0)$ and the positive direction of the real axis.
\end{theorem}
\begin{proof}
This is Theorem 1.7 on p.~174 in~\cite{fedoryuk_book} if we can write $f(x,z) = \eee^{S(x, z)}$ for a function $S(x,z)$ which is analytic in $\Omega_x \times \Omega_z$. The slightly more general case stated above can be reduced to the special case as follows. The problem is that if $\gamma$ is closed, $\log f(x, z)$ might not be well-defined on the whole contour $\gamma$ (even if $x\mapsto f(x, z)$ does not vanish on $\gamma$).  We split the integral $I_n(z)$, taken over the contour $\gamma$, into two parts. The first integral $I_{n;1}(z)$ is taken over the part of $\gamma$ contained in a disk $\bD_{\eps}(\zeta(z_0))$ around $\zeta(z_0)$ which is so small that $x\mapsto f(x, z_0)$ and $x\mapsto g(x, z_0)$ do not vanish on this disk. The second part $I_{n;2}(z)$ is the integral over the remaining part $\gamma\setminus \bD_{\eps}(\zeta(z_0))$. For a sufficiently small $\delta$ it is even true that $f(x, z) \neq 0$ and $g(x, z) \neq 0$ for all $x\in \bD_{\eps}(\zeta(z_0))$ and $z \in \bD_{\delta}(z_0)$, by continuity. We can therefore write $f(x, z) = \eee^{S(x,z)}$ and apply the aforementioned result of~\cite{fedoryuk_book} to get~\eqref{eq:fedoryuk_asymp} with $I_{n;1}(z)$ instead of $I_n(z)$ on the left-hand side. Note that the leading term in this asymptotics is of order $n^{-1/2}(f(\zeta(z),z))^{n}$. The second part of the integral, that is, $I_{n;2}(z)$ can be estimated as follows. If $\rho>0$ is sufficiently small, then the maximum of $x\mapsto |f(x,z)|$ over $\gamma\setminus \bD_{\eps}(\zeta(z_0))$ is smaller than $|f(\zeta(z) , z)| - \rho$, for all $z\in \bD_{\delta}(z_0)$, as a consequence of condition (ii) and continuity arguments. Hence, $|I_{n;2}(z)|$ can be estimated from above by $O((|f(\zeta(z) , z)| - \delta)^n)$ with a constant in $O$-term that does not depend on $z$, which is enough for our purposes.
\end{proof}

\section*{Acknowledgements}
ZK was supported by the German Research Foundation (DFG) under Germany's Excellence Strategy  EXC 2044 -- 390685587, Mathematics M\"unster: Dynamics - Geometry - Structure. AM was supported by the Alexander von Humboldt Foundation. HP was supported by the Research Training Group 1953 of the DFG. We are very grateful to Christoph Th\"{a}le and V.\ A.\ Vatutin for drawing our attention to a number of important references missing in the first version of the manuscript.

\bibliography{touchard_lambert_bib}
\bibliographystyle{plainnat}

\end{document}